\documentclass[draft]{amsart}
\usepackage{hyperref}
\hypersetup{nesting=true,debug=true,naturalnames=true}
\usepackage{graphicx,amssymb}

\usepackage[square,sort&compress,comma,numbers]{natbib}
\usepackage{nicefrac,xcolor,upref,version}

\usepackage{amscd,amsthm,amsmath,amssymb,amsfonts}

\usepackage{mathrsfs}

\newtheorem{theorem}{Theorem}[section]
\newtheorem{lemma}[theorem]{Lemma} 
 
\newtheorem{definition}[theorem]{Definition} 
\newtheorem{problem}[theorem]{Problem} 
\newtheorem{corollary}[theorem]{Corollary}

\numberwithin{equation}{section}

\def\hw{{\widehat w}}

\def\hmu{{\widehat{\mu}}}
\def\hmu{{\widehat{\mu}}}

\def\astmult{{\times}}

\def\tmu{{\mu^{\times}}}
\def\htmu{{\widehat{\mu}^{\times}}}

\def\eps{{\varepsilon}}

\def\cS{{\mathcal S}}

\begin{document}

\title[Exponents of multiplicative $p$-adic approximation]{Classical and uniform exponents of multiplicative $p$-adic approximation}

\author{Yann Bugeaud}
\address{Universit\'e de Strasbourg, Math\'ematiques,
7, rue Ren\'e Descartes, 67084 Strasbourg  (France)}
\email{bugeaud@math.unistra.fr}

\author{Johannes Schleischitz}
\address{Middle East Technical University, Northern Cyprus Campus, Kalkanli, G\"uzelyurt}
\email{johannes@metu.edu.tr}

\begin{abstract}
Let $p$ be a prime number and $\xi$ an irrational $p$-adic number. 
Its irrationality exponent $\mu (\xi)$  is the 
supremum of the real numbers $\mu$ for which the system of inequalities
$$
0 < \max\{|x|, |y|\} \le X, \quad |y \xi - x|_{p} \leq X^{-\hmu}    
$$
has a solution in integers $x, y$ for arbitrarily large real number $X$. 
Its multiplicative irrationality exponent $\tmu (\xi)$ (resp., uniform multiplicative  
irrationality exponent $\htmu (\xi)$)  is the 
supremum of the real numbers $\hmu$ for which the system of inequalities
$$
0 < |x y|^{1/2} \le X, \quad |y \xi - x|_{p} \leq X^{-\hmu}    
$$
has a solution in integers $x, y$ for arbitrarily large 
(resp., for every sufficiently large) real number $X$. It is not difficult to 
show that $\mu (\xi) \le \tmu(\xi) \le 2 \mu (\xi)$ and $\htmu (\xi) \le 4$. 
We establish that the ratio between the multiplicative 
irrationality exponent $\tmu$ and the irrationality exponent $\mu$ can take 
any given value in $[1, 2]$. 
Furthermore, we prove that $\htmu (\xi) \le (5 + \sqrt{5})/2$ 
for every $p$-adic number $\xi$. 
\end{abstract}


\maketitle

\section{Introduction}\label{intro}

Let $\alpha$ be an irrational real number. Its irrationality exponent $\mu (\alpha)$ is the 
supremum of the real numbers $\mu$ for which
\begin{equation}  \label{eq:irr0}
0 < |y \alpha - x|  \leq \max\{ |x|, |y| \}^{-\mu + 1}
\end{equation}  
or, equivalently, 
$$
0 < |\alpha - x/y| \le \max\{ |x|, |y| \}^{-\mu}
$$
has infinitely many solutions in nonzero integers $x, y$. 
Since, for all nonzero integers $x, y$ with $|y \alpha - x| \le 1$, we have
$$
\min\{|x|, |y|\} \ge  \min\{|\alpha|, |\alpha|^{-1}\} \cdot \max \{|x|, |y|\} - 1,
$$
the integers $|x|$ and $|y|$ in \eqref{eq:irr0} have the same order of magnitude and 
we can replace $\max\{ |x|, |y| \}$ in \eqref{eq:irr0} by $|x y|^{1/2}$. 
The same observation does not hold for rational approximation in $p$-adic fields, where similar 
definitions give rise to two different irrationality exponents. 

Throughout this paper, we let $p$ denote a prime number 
and $\mathbb{Q}_{p}$ the field 
of $p$-adic numbers. 
Let $\xi$ be an irrational $p$-adic number. The irrationality exponent $\mu (\xi)$ of $\xi$ is the 
supremum of the real numbers $\mu$ for which
\begin{equation}  \label{eq:irr1}
0 < |y \xi - x|_{p} \leq \max\{ |x|, |y| \}^{-\mu}    
\end{equation}   
has infinitely many solutions in nonzero integers $x, y$. 
Unlike in the real case, the integers $|x|$ and $|y|$ in \eqref{eq:irr1} 
do not necessarily have the same order of magnitude, and one of them can be 
much larger than the other one. 
This has recently been pointed out 
by de Mathan \cite{BdM19}, who studied whether $p$-adic numbers $\xi$ such that 
$$
\inf_{x, y \not= 0} \, |xy| \cdot |y \xi - x|_p > 0
$$
do exist; see also \cite{EiKl07,BaBu21}.\\  

Consequently, it is meaningful to also define the 
multiplicative irrationality exponent $\tmu (\xi)$ of $\xi$ as the 
supremum of the real numbers $\tmu$ for which
\begin{equation}  \label{eq:irr2}
0 < |y \xi - x|_{p} \leq  \bigl(| x y |^{1/2} \bigr)^{-\tmu}    
\end{equation}   
has infinitely many solutions in nonzero integers $x, y$. 
It follows from the Minkowski Theorem (see \cite{Mah34a} or \cite{Mah34b})
and the obvious inequalites
$\max\{|x|, |y|\} \le |xy| \le (\max\{|x|, |y|\})^2$ 
valid for all nonzero integers $x, y$ that we have
\begin{equation}  \label{eq:ineq}
2 \le \mu(\xi) \le \tmu (\xi) \le 2 \mu (\xi). 
\end{equation}   
An easy covering argument shows that  
$\tmu(\xi) = 2$ for almost all $p$-adic number $\xi$. 
Furthermore, the right hand side inequality of 
\eqref{eq:ineq} can be an equality: 
for any sufficiently large integer $c$, 
the $p$-adic number $\xi_c = 1+ \sum_{j \ge 1} p^{c^j}$ 
is well approximated by integers constructed 
by truncating its Hensel expansion and 
it satisfies $\tmu (\xi_c) = 2 \mu (\xi_c)$;
see Theorem \ref{1dspectrum2}. 

Our first results, gathered in Section 2, are concerned with the study of the spectra 
of the exponents of approximation $\mu$ and $\tmu$, 
that is, the set of values taken by these exponents. 
We also investigate the spectrum of their quotient $\tmu / \mu$ and show that it is 
equal to the whole interval $[1, 2]$.

Beside the exponents of approximation $\mu$ and $\tmu$, we 
consider the uniform exponents $\hmu$ and $\htmu$ defined as follows.

\begin{definition}  \label{def1}
Let $\xi$ be an irrational $p$-adic number.  
The uniform irrationality exponent $\hmu (\xi)$ of $\xi$ is the 
supremum of the real numbers $\hmu$ for which the system
\begin{equation}  \label{eq:irr3}
0 < \max\{ |x|, |y| \} \le X, \quad |y \xi - x|_{p} \leq X^{-\hmu}    
\end{equation}   
has a solution for every sufficiently large real number $X$. 
The uniform multiplicative irrationality exponent $\htmu (\xi)$ of $\xi$ is the 
supremum of the real numbers $\htmu$ for which the system
\begin{equation}  \label{eq:irr4}
0 < |x y|^{1/2} \le X, \quad |y \xi - x|_{p} \leq X^{-\htmu}    
\end{equation}   
has a solution for every sufficiently large real number $X$. 
\end{definition}

Let us note that, beside the classical exponent    
(where the points $(x, y)$ belong to a square of area $4 X^2$ centered at the origin) 
and the mulitplicative exponent (where the points $(x, y)$ 
belong to a set of area $16 X^2 (\log X)$ bounded 
by four branches of hyperbola), 
we can as well consider weighted exponents
(where the points $(x, y)$ belong to a rectangle of area $4 X^2$ centered at the origin). 
Although most of our results can be extended to the weighted setting, 
we restrict for simplicity our attention to the
somehow most natural exponents $\tmu$ and $\htmu$ defined above.

We point out that $x$ and $y$ are not assumed to be coprime in \eqref{eq:irr1}, \eqref{eq:irr2}, 
\eqref{eq:irr3}, nor in \eqref{eq:irr4}. Adding this assumption would not change the values of 
$\mu (\xi)$ and $\tmu (\xi)$, but would change the values 
of the uniform exponents at some $p$-adic numbers $\xi$. 
As in the real case, it is not difficult to show that $\hmu (\xi) = 2$ for every irrational 
$p$-adic number $\xi$ (this follows from \cite[Satz 2]{Mah38}, 
see also Lemma~\ref{padicle} below). 
This implies that every irrational $p$-adic number $\xi$ satisfies
\begin{equation}  \label{eq:ineq2}
2 = \hmu (\xi) \le \htmu (\xi) \le 2 \hmu (\xi) = 4. 
\end{equation}   
Furthermore, the example of the $p$-adic numbers $\xi_c$ defined above 
shows that the exponent $\htmu$ takes values 
exceeding $2$; see Theorem \ref{1dspectrum2}. 
Thus, unlike $\hmu$, this exponent is not trivial and deserves to be studied more closely. 
Among several results, stated in Section 3, we prove that $\htmu$ 
is bounded from above by $(5 + \sqrt{5})/2$, thereby improving 
the upper bound $4$ given by \eqref{eq:ineq2}.  Thus, unlike \eqref{eq:ineq}, the 
inequalities \eqref{eq:ineq2} are not best possible. 

Throughout this paper, the numerical constants implied by $\ll$ are always positive and 
depend at most on the prime number $p$.  
Furthermore, the symbol $\asymp$ means that both inequalities $\ll$ and $\gg$ hold.

\section{On the spectra of $\mu^{\astmult}, \hmu^{\astmult}$, and $\mu^{\astmult} / \mu$}

We begin with explicit examples of lacunary Hensel expansions, which include the $p$-adic 
numbers $\xi_c$ defined in Section 1. 

\begin{theorem}  \label{1dspectrum2}
	Let $(a_k)_{k \geq 0}$ be an increasing sequence of non-negative integers with $a_0 = 0$ and 
	$a_{k+1} \ge 2 a_k$ for every sufficiently large integer $k$. Define
	$$
	\xi=\sum_{k=0}^{\infty} p^{a_{k}} = 1 + p^{a_1} + p^{a_2} + \cdots . 
	$$
	Set
	\[
	c=\liminf_{k \to\infty} \frac{a_{k+1}}{a_{k}}, \quad  d = \limsup_{k \to\infty}\frac{a_{k+1}}{a_{k}}, 
	\]
	where $c, d$ are in $[2, + \infty]$. 
	Then, we have
	\begin{equation}  \label{eq:haupta2}
	\mu(\xi)=  d, \quad \mu^{\astmult}(\xi)= 2 d, 
	\end{equation}
	and
	\begin{equation}  \label{eq:2esti}
	3-\frac{1}{c}\; \leq\; \hmu^{\astmult}(\xi) \; \leq \; 3+\frac{1}{d-1}.
	\end{equation}
\end{theorem}

The left hand equality of \eqref{eq:haupta2} has been established in \cite{BuPe18}, the best rational   
approximations being given by the integers $\sum_{j=0}^J p^{a_{j}}$, with $J \ge 1$, obtained 
by truncation of the Hensel expansion of $\xi$. In view of the definition of $\tmu$ and of 
\eqref{eq:ineq}, this implies the right hand equality of \eqref{eq:haupta2}. 
The left hand inequality of \eqref{eq:2esti} is proved in Section 5, while the 
right hand inequality is derived from \eqref{eq:andere} below.  
For small values of $d$, Theorem \ref{endlich} below 
slightly sharpens the right hand inequality of \eqref{eq:2esti}. 
We believe that the left hand inequality in \eqref{eq:2esti} is actually an equality. 

Recall that a $p$-adic Liouville 
	number is, by definition, an irrational 
	 $p$-adic number whose irrationality exponent is infinite. 	
	 The case where $c$ and $d$ are infinite yields the following statement. 

\begin{corollary} \label{cor:ex}
	The $p$-adic Liouville number
	\[
	\xi_{\infty}:= \sum_{j=1}^{\infty} p^{j!}
	\]
	satisfies $\hmu^{\astmult}(\xi_{\infty})=3$. Consequently, 
	the spectrum of $\hmu^{\astmult}$ contains $3$.
\end{corollary}

Inequalities \eqref{eq:ineq} motivate the study of the joint spectrum of the exponents $\mu$ and $\tmu$ and 
of the spectrum of their quotient $\tmu / \mu$, which, by \eqref{eq:ineq}, is included in the interval $[1, 2]$. 

\begin{theorem}  \label{theo}
	For any pair of real numbers $(\mu,\tmu)$ satisfying
	 \begin{equation}  \label{range}
	 \tmu > 5+\sqrt{17}, \quad  \frac{\tmu}{2}   \leq \mu\leq \tmu,
	 \end{equation}
	there exists a $p$-adic number $\xi$ 
	such that $\mu^{\astmult} (\xi)=\tmu$ and
	$\mu(\xi)=\mu$. Consequently, the spectrum of the quotient $\mu^{\astmult} / \mu$ is
	equal to the whole interval $[1,2]$.
\end{theorem}

The restriction $\tmu > 5+\sqrt{17}$ in Theorem \ref{theo} 
comes from the proof and has no reason to be best possible. 
We believe that \eqref{range} can be replaced by the inequalities 
$ \max\{ 2, {\tmu / 2} \} \leq \mu\leq \tmu$.

Let $\dim$ denote the Hausdorff dimension. The $p$-adic analogue 
of the theorem of Jarn\'\i k and Besicovitch \cite{Jar45,Lutz55} asserts that, for every real number $\mu \ge 2$, we have
	\[
	\dim (\{ \xi\in\mathbb{Q}_{p}: \mu (\xi)\geq \mu \}) =
	\dim (\{ \xi\in\mathbb{Q}_{p}: \mu (\xi) = \mu \}) =
	\frac{2}{\mu};
	\]
	see~\cite{BMo} for a more general $p$-adic result.
Combining this result with \eqref{eq:ineq} and an easy 
covering argument, we deduce that 
	\[
	\dim (\{ \xi\in\mathbb{Q}_{p}: \mu^{\astmult}(\xi)\geq \tmu \}) =
	\dim (\{ \xi\in\mathbb{Q}_{p}: \mu^{\astmult}(\xi) = \tmu \}) =
	\frac{2}{\tmu} 
	\]
holds for every real number $\tmu \ge 2$. Consequently, the spectrum of $\mu^{\astmult}$ 
is equal to the whole interval $[2, +\infty]$. It would be interesting to construct explicitly, for any 
real number $\tmu \ge 2$, a $p$-adic number $\xi_\tmu$ satisfying $\tmu (\xi_\tmu) = \tmu$. For $\tmu \ge 4$, 
such examples are given in Theorem \ref{1dspectrum2}. 

\begin{problem}  \label{pro2.4}
For $\tmu$ any real number with $2 \le \tmu < 4$, construct explicitly a $p$-adic 
number $\xi_\tmu$ such that $\tmu (\xi_\tmu) = \tmu$. 
\end{problem} 

The algorithm presented in Section~\ref{auxres} below may be helpful
for answering Problem~\ref{pro2.4}. However,
if we impose the additional natural condition 
$\mu^{\times}(\xi_\tmu)> \mu(\xi_\tmu)$,
new difficulties occur.

In a subsequent work, we will study more closely the 
classical and uniform multiplicative exponents of $p$-adic numbers 
whose Hensel expansion is given by a classical combinatorial sequence, like the 
Thue--Morse sequence or a Sturmian sequence. Let us just note that the $p$-adic 
Thue--Morse number
$$
\xi_{TM} = 1 + p^3 + p^5 + p^6 + p^9 + p^{10} + \ldots
$$
satisfies $\mu (\xi_{TM}) = 2$ (see \cite{BuYao17}) and $\tmu (\xi_{TM}) \ge 3$, 
where presumably this inequality is in fact an equality.

\section{Upper bounds for the uniform exponent $\hmu^{\astmult}$}

In the main result of this section, we improve the trivial upper bound $4$ given in \eqref{eq:ineq2} 
for the exponent of uniform approximation $\htmu$.

\begin{theorem} \label{endlich}
	Any irrational $p$-adic number $\xi$ satisfies
	\begin{equation}  \label{eq:andere}
	\hmu^{\astmult}(\xi) \leq 3 +\frac{2}{\mu^{\astmult}(\xi)- 2},
	\end{equation}
	\begin{equation}  \label{eq:jnik}
	\mu^{\astmult}(\xi) \geq \hmu^{\astmult}(\xi)^2 - 
	3 \hmu^{\astmult}(\xi) + 3, 
	\end{equation}
	and
		\begin{equation}  \label{eq:bndere}
	\hmu^{\astmult}(\xi) \leq  \frac{5+\sqrt{5}}{2} = 3.6180\ldots 
	\end{equation}
\end{theorem}

The first assertion of Theorem~\ref{endlich} is stronger than the third one
only when $\mu^{\astmult}(\xi)$ exceeds $3 +\sqrt{5} = 5.23\ldots$.

The combination of 
\eqref{eq:andere} and \eqref{eq:jnik} gives
$$
\hmu^{\astmult}(\xi) \leq 3 +\frac{2}{\hmu^{\astmult}(\xi)^2 - 3 \hmu^{\astmult}(\xi) + 1},
$$
thus
$$
(\hmu^{\astmult}(\xi)  - 1) \bigl( \hmu^{\astmult}(\xi)^2 - 5 \hmu^{\astmult}(\xi) + 5 \bigr) \le 0, 
$$
and we obtain \eqref{eq:bndere}. Therefore, to establish Theorem \ref{endlich}, it is 
sufficient to prove \eqref{eq:andere} and \eqref{eq:jnik}.

Note that \eqref{eq:jnik} is of interest only for
putative $\xi$ with $\hmu^{\astmult}(\xi)> 3$. 
Combined with \eqref{eq:ineq} it implies 
that 
	$$  
	\mu(\xi)\geq \frac{\hmu^{\astmult}(\xi)^2 - 3 \hmu^{\astmult}(\xi) + 3}{2}.
	$$ 
In particular,  if $\hmu^{\astmult}(\xi)>(3+\sqrt{13})/2=3.3027\ldots$ then $\mu(\xi)>2$,
	thus, $\xi$ is very well approximable. In other words, if
	$\mu(\xi)=2$ then $\hmu^{\astmult}(\xi)\leq (3+\sqrt{13})/2$.

We display an immediate consequence of \eqref{eq:andere}. 

\begin{corollary} \label{cor:Liouv}
	Any $p$-adic Liouville number $\xi$ satisfies
	$
	\hmu^{\astmult}(\xi) \leq 3.
	$
\end{corollary}

In view of Corollary \ref{cor:ex}, the upper bound $3$ for Liouville 
numbers obtained in Corollary \ref{cor:Liouv} is best possible. 
We cannot exclude that $\htmu$ is always bounded by $3$. 

For the proof of Theorem~\ref{endlich}, we introduce the sequence
$(x_{k}^{\astmult}, y_{k}^{\astmult})_{k \ge 1}$ of multiplicative
best approximations to $\xi$, defined in Section~\ref{se8}.
We are able to get the stronger conclusion $\hmu^{\astmult}(\xi) \leq 3$ 
under certain conditions.

\begin{theorem}  \label{neu}
Assume that at least one of the following two claims holds

\smallskip

(i) There exist $c > 0$ and are arbitrarily large $k$ such that 
$|x_{k}^{\astmult}|\ge c |y_{k}^{\astmult}|$ and $|x_{k+1}^{\astmult}|\ge c |y_{k+1}^{\astmult}|$; 

(ii) There exist $c > 0$ and are arbitrarily large $k$ such that 
$|x_{k}^{\astmult}|\le c |y_{k}^{\astmult}|$ and $|x_{k+1}^{\astmult}|\le c |y_{k+1}^{\astmult}|$. 
\smallskip

Then we have $\hmu^{\astmult}(\xi) \leq 3$.
\end{theorem}

The upper bound $\frac{5+\sqrt{5}}{2}$ in Theorem \ref{endlich} 
is obtained when, simultaneously, $|x_{2k}^{\astmult}|$ is very small 
compared to $|y_{2k}^{\astmult}|$ and $|x_{2k + 1}^{\astmult}|$ is very large
compared to $|y_{2k + 1}^{\astmult}|$, 
or vice versa, 
for every sufficiently large integer $k$. 
We cannot exclude the existence of a $p$-adic number whose sequence of multiplicative
best approximations has this property.

The main difference with the classical setting 
occurs when we estimate the $p$-adic value of the difference between 
distinct rational numbers. Let $x, y, x', y'$ be nonzero integers, not divisible by $p$ and such that 
$ x y' \not= x' y$. 
Then, $|x/y - x'/y'|_p^{-1} = | x y' - x' y|_p^{-1}$ is at most equal to $|x y'| + |x' y|$, which 
can be much larger than the product $| x y|^{1/2}$ times $|x' y'|^{1/2}$, in particular when 
simultaneously $|x|$ is much larger than $|y|$ and $|y'|$ is much larger than $|x'|$. Thus, we 
cannot avoid to use the trivial estimate $| x y' - x' y|_p^{-1} \le 2 \max\{|x|, |y|\}  \max\{|x'|, |y'|\}$, 
which involves the sup norm.

It follows from \eqref{eq:jnik} that	
	\[
	\dim (\{\xi\in\mathbb{Q}_{p}: \hmu^{\astmult}(\xi) \geq \tmu \})
	\leq \frac{2}{(\tmu)^2 - 3 \tmu + 3 }, \quad \tmu\in \biggl[3, \frac{5+\sqrt{5}}{2} \biggr].
	\]

Our results motivate the following question.

\begin{problem}
Determine the Hausdorff dimension of the sets
\[
\{ \xi\in\mathbb{Q}_{p}: \hmu^{\astmult}(\xi) \geq \tmu \}, \quad
\{ \xi\in\mathbb{Q}_{p}: \hmu^{\astmult}(\xi) = \tmu \}, \quad  \tmu\in \biggl[2, \frac{5+\sqrt{5}}{2} \biggr]. 
\]
\end{problem}

We end this section with a remark.        
It follows from Theorem \ref{endlich} that any $p$-adic number $\xi$ 
        with $\htmu (\xi) = \frac{5+\sqrt{5}}{2}$ also satisfies
	$
	\mu^{\astmult}(\xi)= 3 +\sqrt{5}.	 
	$
         A similar situation occurs with the extremal numbers defined by Roy \cite{Roy04}. 
         These are transcendental real numbers $\alpha$ 
         whose uniform exponent of quadratic approximation takes the
         maximal possible value, that is, for which we have $\hw_2 (\alpha) = (3 + \sqrt{5})/2$. 
         Roy \cite{Roy04} proved that they satisfy
         \begin{equation} \label{eq:extr} 
	1 + w_2^* (\alpha) = 3 +\sqrt{5}, \quad 
	1 + \hw_2 (\alpha) = \frac{5 +\sqrt{5}}{2},
	\end{equation} 
	where $w_2^*$ and $\hw_2$ denote classical and uniform exponents of quadratic approximation.
	Subsequently, Moshchevitin \cite{Mo14} established that every irrational, non-quadratic 
	real number $\alpha$ satisfies
	$$
	w_2^* (\alpha) \ge \hw_2 (\alpha) (\hw_2 (\alpha) - 1), 
	$$
	\begin{equation} \label{eq:mosh} 
	1 + w_2^* (\alpha) \ge (1 + \hw_2 (\alpha) )^2 - 3 (1 + \hw_2 (\alpha) ) + 3,
	\end{equation} 
	with equality when $\alpha$ satisfies \eqref{eq:extr}. 
	Furthermore, by \cite[Ineq. (2.5)]{buschlei}, we also have
	\begin{equation} \label{eq:bs} 
	1 + \hw_2 (\alpha) \le 3 + \frac{2}{(1 + w_2^* (\alpha) ) - 2},
	\end{equation} 
	with equality when $\alpha$ satisfies \eqref{eq:extr}. 
	Since \eqref{eq:mosh} and \eqref{eq:bs} are analogous to \eqref{eq:jnik} and \eqref{eq:andere}, 
	respectively, this may suggest that the bounds of Theorem~\ref{endlich} are best possible.

\section{Preparatory results}  \label{prepres}

First, we observe that in the definitions of the exponents of approximation $\mu$ and 
$\mu^{\astmult}$, we can assume that the integers $x$ and $y$ are coprime. 
This is not the case for the uniform exponents. 

The next two statements are $p$-adic analogues of classical results 
in the real case, which can be easily proved using the theory of continued fractions. They can most likely be found
in the literature, but we choose to supply short proofs for the 
convenience of the reader.

\begin{lemma}  \label{padicle}
	Let $\xi$ be in $\mathbb{Q}_{p}$. 
	There do not exist two linearly independent integer pairs 
	$(x_{1},y_{1})$ and $(x_{2},y_{2})$, which, setting $X_{i}=\max \{ |x_{i}|,|y_{i}|\}$ for $i=1,2$, 
	satisfy
	\[
	|y_{i}\xi -x_{i}|_{p} < \frac{1}{2}X_{1}^{-1} X_{2}^{-1}, \quad i=1,2.
	\]
	In particular, for any real number $X>1$, the system
	\[
	\max\{ |x|,|y|\}\leq X, \quad |y \xi - x|_{p} < \frac{1}{2}X^{-2}, 
	\]
	does not have two linearly independent integer solutions.
\end{lemma}

Lemma \ref{padicle} easily implies that every irrational $p$-adic number $\xi$
satisfies $\hmu (\xi) \le 2$, thus $\hmu (\xi) = 2$,
a fact already stated in Section~\ref{intro}. 

\begin{proof}
	Assume, on the contrary, that there are two linearly independent 
	pairs of solutions $(x_1, y_1)$ and $(x_2, y_2)$ as in the statement of
	the lemma. 
	Write  
	$X=\max \{X_{1},X_{2}\}$. 
	It follows from the identity
	\[
	x_{1}y_{2}-x_{2}y_{1}= y_{1}(y_{2} \xi -  x_{2}) - y_{2}(y_1 \xi -  x_{1}) 
	\]
	that
	$$
	|x_{1}y_{2}-x_{2}y_{1}|_{p} \leq \max\{ |y_{2} \xi -  x_{2}|_{p} \; , \; |y_1 \xi -  x_{1}|_{p}  \} 
	< \frac{1}{2X_{1}X_{2}}.
	$$
	Since $x_{1}y_{2} \not= x_{2}y_{1}$, we get
	\[
	|x_{1}y_{2}-x_{2}y_{1}|_{p}\geq \frac{1}{|x_{1}y_{2}-x_{2}y_{1}|}\geq \frac{1}{|x_{1}y_{2}|+|x_{2}y_{1}|}\geq \frac{1}{2X_{1}X_{2}},
	\]
	a contradiction. The next claim follows
	directly since $X_{i}\leq X$ for $i=1, 2$ implies that $X^{-2}\leq X_{1}^{-1}X_{2}^{-1}$.
\end{proof}

For an irrational $p$-adic number $\xi$ and a real number $T\geq 1$, 
let $(x(T),y(T))$ denote the pair of coprime integers 
which minimizes $|y(T)\xi- x(T)|_{p}$ among all
the pairs $(x,y)$ of coprime integers with $\max\{ |x|,|y|\} \leq T$.
As $T$ increases, this gives rise to a sequence of best approximations 
$(x_{j},y_{j})$, $j \ge 1$, to $\xi$ with the properties
\[
\max\{ |x_{1}|,|y_{1}|\}< 
\max\{ |x_{2}|,|y_{2}|\}<\cdots, \quad |y_{1}\xi-x_{1}|_{p}>|y_{2}\xi-x_{2}|_{p}>\cdots.
\]
and $|y_{j}\xi-x_{j}|_{p}$ minimizes $|y \xi-x|_{p}$ among all the pairs $(x,y)$ of coprime integers with
$\max\{ |x|,|y|\} \leq \max\{ |x_{j}|,|y_{j}|\}$.

\begin{corollary}  \label{korollar}
	If $(x_{k},y_{k})$ is a best approximation to $\xi$ in $\mathbb{Q}_{p}$ and $\tau_{k}$ is defined by  
	\[
	|y_{k}\xi - x_{k}|_{p} = \max\{ |x_{k}|,|y_{k}|\}^{-\tau_{k}},
	\]
	then the next best approximation $(x_{k+1},y_{k+1})$ to $\xi$ 
	satisfies
	\[
	\frac{1}{2}\max\{ |x_{k}|,|y_{k}|\}^{\tau_{k}-1} \leq  \max\{ |x_{k+1}|,|y_{k+1}|\}\leq 
	(p+1) \max\{ |x_{k}|,|y_{k}|\}^{\tau_{k}-1}. 
	\]
	Thus, 
	\[
	|y_{k}\xi - x_{k}|_{p} \asymp \max\{ |x_{k}|,|y_{k}|\}^{-1}\max\{ |x_{k+1}|,|y_{k+1}|\}^{-1}.
	\]
\end{corollary}

\begin{proof}
	Set $Q_{j}=\max \{ |x_{j}|,|y_{j}|\}$ for $j \geq 1$.
	Assume that $Q_{k+1}< C Q_{k}^{\tau_{k}-1}$ for some $C<\frac{1}{2}$.
	Since any two best approximations are linearly independent, the inequalities
	\[
	|y_{k+1} \xi - x_{k+1}|_{p}< |y_{k}\xi - x_{k}|_{p}=Q_{k}^{-\tau_{k}}<CQ_{k}^{-1}Q_{k+1}^{-1}
	<\frac{1}{2}Q_{k}^{-1}Q_{k+1}^{-1},
	\]
	contradict Lemma~\ref{padicle}. This proves the left hand estimate. 
	
	For the right hand estimate notice that, 
	by \cite[Satz 1]{Mah34a}, for every positive integer $h$, there are integers $x, y$,
	not both zero, such that $|y \xi - x|_p \le p^{-h}$ and $\max\{|x|, |y|\} \le p^{h/2}$. 
	Consequently, for every $Q$ there is a pair $(x,y)$ of integers with $\max\{|x|,|y|\}\leq Q$ 
	and $|y \xi -x|_{p}\leq p Q^{-2}$. 
	
	Set $Q=(p+1) \max\{ |x_{k}|,|y_{k}|\}^{\tau_{k}-1}$. 
	For any positive integer $M$ satisfying $M\cdot\max\{ |x_{k}|,|y_{k}|\} \le Q$ we have
	\[
	|My_{k}\xi -M x_{k}|_{p}\geq M^{-1}|y_{k}\xi -x_{k}|_{p}\geq
	\frac{\max\{ |x_{k}|,|y_{k}|\}}{Q}\cdot \max\{ |x_{k}|,|y_{k}|\}^{-\tau_{k}}> p Q^{-2}.
	\]
	This implies that $\max\{ |x_{k+1}|,|y_{k+1}|\}\leq Q$ and completes the proof.
	\end{proof}

\section{Proof of Theorem~\ref{1dspectrum2}}  \label{seunipara}

\begin{proof}[Proof of Theorem~\ref{1dspectrum2}]

	Let $\xi$ be as in the theorem and define the rational integers
	\[
	Q_{k}= \sum_{j=0}^{k}  p^{a_{j}}, \quad k \ge 1. 
	\]
	Then with $c_{k}= a_{k+1}/ a_{k}$ for $k \ge 1$ we get
	\[
	|\xi-Q_{k}|_{p} = \Bigl|\sum_{j=k+1}^{\infty}  p^{a_{j}} \Bigr|_{p}
	\asymp p^{-a_{k+1}}
	\asymp Q_{k+1}^{-1} \asymp Q_{k}^{-c_{k}}, \quad k \ge 1. 
	\]
		This in particular shows that $\mu(\xi) \geq d$ and $\tmu(\xi) \geq 2 d$, since
		there are arbitrarily large $k$ such that $c_k$  is
		arbitrarily close to $d$. The equality $\mu(\xi) = d$ has been established in~\cite{BuPe18}. 
		By \eqref{eq:ineq}, this gives $\tmu(\xi) \leq 2 d$ and proves \eqref{eq:haupta2}. 
	Set $Q_{k}^{\astmult} = \sqrt{Q_{k}}$ for $k \geq 1$. 
	For a given integer $X$, let
	$k$ be the index defined by $Q_{k}^{\astmult}\leq X< Q_{k+1}^{\astmult}$. 
	We then have
	\[
	Q_{k}^{\astmult}  = \sqrt{Q_{k}\cdot 1} \leq X, \quad |\xi-Q_{k}|_{p} \asymp  (Q_{k}^{\astmult})^{ -2c_{k} }.
	\]
	Let $M$ be the largest integral power of $p$ smaller than $X/Q_{k}^{\astmult}$.
	Then
	\[
	\sqrt{(M\cdot 1) \cdot (M\cdot Q_{k})} \leq X  
	\]
	and
	\[
	|M\xi-MQ_{k}|_{p} \ll M^{-1}(Q_{k}^{\astmult})^{ -2c_{k} }\ll \frac{Q_{k}^{\astmult}}{X}(Q_{k}^{\astmult})^{ -2c_{k} }\ll X^{\frac{1-2c_{k} }{c_{k} } -1}= X^{-3+\frac{1}{c_{k}} } .
	\]
	By definition of $c$, we get the lower bound $\hmu^{\astmult}(\xi)\geq 3-1/c$ in \eqref{eq:2esti}. 
	The upper bound follows from \eqref{eq:andere}.
	\end{proof}

\section{An auxiliary result}  \label{auxres}

The proof of Theorem~\ref{theo} is semi-constructive and uses

\begin{theorem}  \label{hopefully}
	For any $\tilde{\mu}>2$ and any $\epsilon>0$,
	there exists $\xi$ in $\mathbb{Z}_{p}$ with the following properties.
	There exists a sequence 
	$((x_{j,0}, x_{j,1}))_{j \ge 1}$ of pairs of coprime integers not divisible by $p$, whose moduli 
     tend to infinity, 
	and satisfy the following properties:
	\begin{itemize}
		\item 	We have
		\[
		|x_{j,0}|\asymp |x_{j,1}|, \quad  |x_{j,1}\xi - x_{j,0}|_{p} 
		\asymp |x_{j,0}|^{- \tilde{\mu} }\asymp |x_{j,1}|^{- \tilde{\mu} }, \quad j\geq 1,
		\]
		\item We have
		\[
		\lim_{j\to\infty} \frac{\log |x_{j+1,0}|}{\log |x_{j,0}|}= \infty
		\]
		\item For every integer pair $(z_{0},z_{1})$ linearly
		independent of any pair $(x_{j,0},x_{j,1})$ with $j \ge 1$,
		we have
		\begin{equation}  \label{eq:z0z1}
		|z_{1}\xi - z_{0}|_{p} \gg \max\{ |z_{0}|, |z_{1}|\}^{-2-\epsilon}.
		\end{equation}
	\end{itemize} 
\end{theorem}

The first property implies $\mu(\xi) \ge \tilde{\mu}$.
The second property states that there are large gaps between consecutive very good approximations. 
The third states that at most finitely of the other approximations are very 
good, thus (when $\epsilon<\mu-2$) we have $\mu(\xi) = \tilde{\mu}$.
It follows that
$((x_{j,0}, x_{j,1}))_{j \ge 1}$ is a subsequence 
of the sequence of best approximations defined in Section~\ref{prepres},
but may not contain all of them. We remark that \eqref{eq:z0z1}
may be sharpened, indeed, using refined estimates 
the proof actually yields the lower bound 
$\gg_{\varepsilon} Z^{-2} (\log Z)^{-1-\varepsilon}$, 
where $Z=\max\{ |z_{0}|, |z_{1}|\}$ and $\varepsilon>0$ can be taken arbitrarily small.

\begin{proof}[Preparation for the proof of Theorem~\ref{hopefully}]
Fix $\epsilon>0$. 
We first construct a $p$-adic number $\xi$ in such a way that we control 
the quality of its best rational approximations, apart possibly 
some good approximations $(x, y)$, for which $|y \xi - x|\gg \max\{ |x |, |y |\}^{-2-\epsilon}$. 

More precisely, for a given sequence $(\mu_{n})_{n\geq 1}$
with $\mu_{n}\geq 2+\epsilon$ for $n \ge 1$, 
we find $\xi$ as the $p$-adic limit of a sequence
of rationals $p_{n}/q_{n}$ with $p_n\asymp q_n$ 
and upon writing
\[
L_{i}= |p_{i}q_{i+1}-p_{i+1}q_{i}|_{p}, \quad H_i = \max\{|p_i|, |q_i|\}, \quad i \ge 0, 
\]
we have
\begin{equation}  \label{eq:indeed}
H_n^{-\mu_n} \leq L_n \leq p H_n^{-\mu_n}, \quad n\geq 1,
\end{equation}
and \eqref{eq:z0z1} holds
for any $(z_0,z_1)$ linearly independent of all $(p_{n},q_{n})$.

We construct a sequence $p_{n}/q_{n}$ that will converge at some
given rate with respect to the $p$-metric to some $p$-adic number $\xi$. 
We use Schneider's continued fraction algorithm; see e.g. \cite{Bu10}. 
Start with 
\[
p_{-1}=1,\; q_{-1}=0, \quad p_{0}=0, \; q_{0}=1.
\]
Then $|p_{-1}q_{0}-q_{-1}p_{0}|=1$.
Then recursively let
\begin{equation}  \label{eq:rekursion}
p_{n+1}= p_{n}+b_{n+1}p_{n-1}, \quad q_{n+1}=q_{n}+b_{n+1}q_{n-1}
\end{equation}
where each $b_{n}=p^{g_{n}}$ is an integer 
power of $p$ to be chosen at any step accordingly.
For all $n$, since $|b_{n}|_{p}=b_{n}^{-1}$ we
calculate
\begin{align*}
|p_{n}q_{n+1}-p_{n+1}q_{n}|_{p} &= |p_{n}(q_{n}+b_{n+1}q_{n-1})-q_{n}(p_{n}+b_{n+1}p_{n-1})|_{p} 
 \\
&=|b_{n+1}(p_{n}q_{n-1}-p_{n-1}q_{n})|_{p}= 
\frac{1}{b_{n+1}}\cdot |p_{n-1}q_{n}-p_{n}q_{n-1}|_{p}.  
\end{align*}
Setting $L_{i}= |p_{i}q_{i+1}-p_{i+1}q_{i}|_{p}$ and $H_i = \max\{|p_i|, |q_i|\}$ for $i\geq 0$ as above, 
we see that
\begin{equation}  \label{eq:Landb}
L_{n}= \frac{1}{b_{n+1}}\cdot L_{n-1}, \quad n \ge 1. 
\end{equation}
It is easy to see that all $p_{i},q_{i}$ are not divisible by $p$, hence
\[
 \Bigl|\frac{p_{n+1}}{q_{n+1}}-\frac{p_{n}}{q_{n}} \Bigr|_{p}=|p_{n}q_{n+1}-p_{n+1}q_{n}|_{p}
= \frac{1}{b_{n+1}}\cdot|p_{n-1}q_{n}-p_{n}q_{n-1}|_{p}. 
\]
Since $b_{n}\geq p>1$, the rational numbers
$p_{n}/q_{n}$ form a Cauchy sequence and thus
converge with respect to the $p$-adic metric to some $p$-adic number $\xi$. 
Observe that
\[
|q_{n}\xi-p_{n}|_{p}= \Bigl|\xi-\frac{p_{n}}{q_{n}} \Bigr|_{p} 
= \Bigl|\frac{p_{n+1}}{q_{n+1}}-\frac{p_{n}}{q_{n}}\Bigr|_{p}
=|p_{n+1}q_{n}-p_{n}q_{n+1}|_{p}, 
\quad n\geq 1, 
\]
where the second identity holds because obviously 
\[
\Bigl|\xi-\frac{p_{n+1}}{q_{n+1}} \Bigr|_{p} < \Bigl|\xi-\frac{p_{n}}{q_{n}} \Bigr|_{p}. 
\]

We get from \eqref{eq:rekursion} for all $n$ that 
\begin{equation}  \label{eq:maximum}
H_{n+1} \leq H_n + b_{n+1}H_{n-1} \leq 2 \max\{ H_{n}, b_{n+1}H_{n-1}\}.
\end{equation}

Assume now that for some fixed integer $N$ we have 
constructed $p_{1}/q_{1}, \ldots,p_{N}/q_{N}$
with the desired approximation properties.
We describe how to choose
$b_{N+1}$ (or $g_{N+1}$) to get the next $p_{N+1}/q_{N+1}$.
Set $\gamma_N= L_{N-1}H_{N}H_{N-1}$ and observe that the inequality
\begin{equation} \label{eq:annahme}
L_{N-1}\leq \gamma_N \cdot H_{N}^{-1}H_{N-1}^{-1}.
\end{equation}
holds. Now,  define recursively
\[
g_{n+1} = \left\lfloor  \frac{ \log H_n^{\mu_n} L_{n-1} }{\log p}  \right\rfloor,
\quad n\geq N,
\]
which is the largest integer with $b_{n+1}=p^{g_{n+1}}\leq H_n^{\mu_n} L_{n-1}$.
We readily conclude from \eqref{eq:Landb}
that then indeed \eqref{eq:indeed} holds for 
all $n\geq 1$.

By an easy induction, it follows from \eqref{eq:rekursion} 
that $p_k\asymp q_k$ for $k\geq 1$. Moreover it is clear from
the recursion \eqref{eq:rekursion} that $p_n$ and $q_n$ are 
coprime for all $n$.

It remains to be shown that there is no good approximation
apart from the $p_n/q_n$, i.e., that \eqref{eq:z0z1} holds. For 
this we first estimate the growth of the height sequence $(H_{n})_{n\geq 1}$.
For the initial value $n=N$, in case the maximum in \eqref{eq:maximum} is $b_{n+1}H_{n-1}=b_{N+1} H_{N-1}$, 
by \eqref{eq:Landb} and \eqref{eq:indeed} we have
\begin{align} \label{eq:nextstep}
H_{N+1} \leq  2b_{N+1}H_{N-1}= 2\frac{L_{N-1}}{L_N}H_{N-1} &\leq 2L_{N-1} H_N^{\mu_N} H_{N-1}
\nonumber \\& = 2H_N^{\mu_N}(L_{N-1}H_{N-1})\leq 
2\gamma_N H_{N}^{\mu_{N}-1}, 
\end{align}
where we used our induction assumption \eqref{eq:annahme}.
Then by \eqref{eq:nextstep} in view of \eqref{eq:indeed} for $n$
(which we verified above) we have
\[
L_{N}\leq p H_{N}^{-\mu_{N}}\leq 2p\gamma_{N}\cdot  H_{N+1}^{-1}H_{N}^{-1}.
\]
In other words, in the next step similar to \eqref{eq:annahme} we have
\[
L_{N}\leq \gamma_{N+1} H_{N+1}^{-1}H_{N}^{-1}, \quad 
\gamma_{N+1}=2p \gamma_{N}.
\]
Thus similarly as in \eqref{eq:nextstep} above we infer
\[
H_{N+2} \leq 2\gamma_{N+1} H_{N+1}^{\mu_{N+1}-1}.
\]
Iterating this process we see that for all $n\geq N$ 
and some fixed $c>1$ we get
\[
L_{n}\leq  \gamma_N(2p)^n\cdot  H_{n+1}^{-1}H_{n}^{-1}\ll c^{n}\cdot  H_{n+1}^{-1}H_{n}^{-1}
\]
and
\begin{equation}  \label{eq:H}
H_{n} \leq \gamma_N(4p)^n\cdot H_{n-1}^{\mu_{n-1}-1}\ll c^{n}\cdot H_{n-1}^{\mu_{n-1}-1}.
\end{equation}

Otherwise, if the maximum in \eqref{eq:maximum} is $H_{N+1}$, 
then we directly get 
\[
H_{N+1}\leq  2H_{N}\leq 2H_{N}^{\mu_{N}-1},
\]
since $\mu_{n}\geq 2+\epsilon>2$, which is even stronger 
than the estimates derived in the other case and we infer the same result.

Notice that by Corollary~\ref{korollar} we have 
$H_{n+1}\geq H_n^{ \min\{ \mu_n , 2+\epsilon \} -1 }/2 = H_n^{1+\epsilon}/2$
for all large $n$. Thus the sequence $(\log H_{n})_{n \ge 2}$ grows  exponentially fast, 
so in particular 
$$ 
c^{k}= H_{k}^{o(1)}, \quad k\to\infty .
$$ 
It then follows from \eqref{eq:H} that 
\begin{equation} \label{eq:epsin}
L_{k}\ll_{\varepsilon} H_{k+1}^{-1+\varepsilon}H_{k}^{-1+\varepsilon}, 
\quad H_{k}\ll_{\varepsilon} H_{k-1}^{\mu_{k-1}-1+\varepsilon}, \quad k \ge 2,
\end{equation}
for every $\varepsilon>0$
and some implicit positive constants depending only
on $\varepsilon$.

Now assume that an integer pair $(z_0, z_1)$ is not among the $(p_n,q_n)$ and satisfies 
\begin{equation} \label{eq:viola}
| z_1 \xi - z_0| \leq Z^{-2-\epsilon/2},
\end{equation}
where $Z=\max\{ |z_0|, |z_1| \}$. 
We may assume that $(z_0, z_1)$ is
linearly independent to all $(p_n,q_n)$ 
and that $z_0$ and $z_1$ are coprime. 
Let $k$ be the index 
with $H_{k-1}\leq Z< H_{k}$. By coprimality and 
since 
$$
|q_n\xi-p_n|\leq \max\{ p_n, q_n\}^{-2-\epsilon} < p_{n}^{-1}q_{n}^{-1}/2, 
$$
Lemma~\ref{padicle} implies that 
the pairs $(p_n,q_n)$ 
are best approximations. Similarly, by ${\rm gcd} (z_0,z_1)=1$ and \eqref{eq:viola} 
similarly the pair $(z_0,z_1)$ is a 
best approximation as well.
Thus, on the one hand, by Corollary~\ref{korollar} we have
\[
Z\geq \frac{H_{k-1}^{\mu_{k-1}-1}}{2}, \quad k\geq k_0,
\]
combined with \eqref{eq:epsin} and $Z<H_k$ we get
\[
\frac{\log H_{k}}{\log Z} \leq 1+\eta ,
\]
for arbitrarily small $\eta >0$ and large enough $k$.
On the other hand, again by \eqref{eq:viola} and 
Corollary~\ref{korollar} we must have 
\[
H_{k}\geq \frac{Z^{1+\epsilon/2}}{2}.
\]
By choosing $\eta=\epsilon/3$, we end up with a contradiction for large $k$. 
Thus, \eqref{eq:viola}  cannot hold if $Z$ is large enough. 
\end{proof}

\begin{proof}[Completion of the proof of Theorem~\ref{hopefully}]
	We choose for $(\mu_{n})_{n \ge 1}$ the sequence 
	\[
	2+\epsilon,\tilde{\mu},2+\epsilon,2+\epsilon,\ldots,2+\epsilon,\tilde{\mu},2+\epsilon,2+\epsilon,\ldots , 
	\]
	 with 
	very long blocks of $2+\epsilon$ separating 
	two occurrences of $\tilde{\mu}$. We identify 
	$x_{j,1}=q_{\sigma(j)}$ and $x_{j,0}=p_{\sigma(j)}$ for all $j\geq 1$
	where the injective map $\sigma: \mathbb{N} \to \mathbb{N}$ 
	is defined so that $\sigma(j)$
	is the $j$-th index where $\mu_n=\tilde{\mu}$. 
	The property
	${\rm gcd} (p, x_{j,0}  x_{j,1}) = 1 $ holds since we noticed  
	that ${\rm gcd} (p, p_{i} q_i) = 1$ for all $i$.   
	Moreover, the large gaps guarantee the second claim of the theorem.
	It then follows from the observations above that ${\rm gcd}(x_{j,0}, x_{j,1})=1$ 
	and $x_{j,0}\asymp x_{j,1}$ for all $j$,
	and the estimate $|x_{j,1} \xi- x_{j,0}|_p \asymp x_{j,0}^{-\tilde{\mu}}$ is 
	immediate from \eqref{eq:indeed}. 
	For the remaining $p_n/q_n$ with $n$ not in the image of $\sigma$, 
	we have $\mu_n=2+\epsilon$, and the estimate
	\eqref{eq:z0z1} is implied by \eqref{eq:indeed} again. For all other pairs 
	$(z_0,z_1)$ we have already shown that
	\eqref{eq:viola} does not hold if $\max\{|z_0|, |z_1|\}$ is large enough. 
	This completes the proof. 
		\end{proof}

\section{Proof of Theorem~\ref{theo}}

We prove our Theorem~\ref{theo} using a similar strategy as in~\cite[Theorem~3.7]{schleinew}.
The idea is to start with a $p$-adic number $\zeta$ given by Theorem~\ref{hopefully} and to 
change its Hensel expansion by replacing its digits 
by $0$ in certain large intervals $J_{i}$ in order to obtain a $p$-adic number $\xi$ with the 
requested properties. 
This will induce good
{\it integer} approximations to $\xi$ 
and thereby imply that $\mu^{\astmult}(\xi)$ is rather large. 
We will see that the good approximations $x_{j,0}/x_{j,1}$ 
to $\zeta$ give rise to equally good rational
approximations $y_{j,0}/y_{j,1}$ to $\xi$, thereby showing
$\mu(\xi)\geq \mu$ as well.
The most technical part
is to show that there are no better rational approximations, that is, to verify
the upper bounds 
$\mu(\xi)\leq \mu, \mu^{\astmult}(\xi)\leq \mu^{\astmult}$. 
Here we essentially
use the method invented in~\cite[Theorem~3.7]{schleinew} to show
that putative good approximations to $\xi$ would induce good approximations
to $\zeta$ which are not among the $x_{j,0}/x_{j,1}$, in contradiction 
with Theorem~\ref{hopefully}. In the proof below,
all $\varepsilon_j$ are positive and can be taken arbitrarily close to $0$. 

\begin{proof}[Proof of Theorem~\ref{theo}]
	Fix $t$ in $[1, 2]$ and $\mu>2$. 
	Let $\zeta$ be in $\mathbb{Z}_{p}$ which satisfies the hypotheses of 
	Theorem~\ref{hopefully} for small enough $\epsilon>0$ 
	depending on $\mu$ (this will be made more precise later) and with 
	\[
	\mu(\zeta) = \tilde{\mu} = t \mu.
	\]
	Let $(x_{j,0}, x_{j,1})_{j \ge 1}$ denote the sequence of integer pairs given by Theorem~\ref{hopefully}. 
	Without loss of generality, we
	assume $x_{j,0}>0$ for $j \ge 1$ and that $x_{1,0}$ and $|x_{1,1}|$ are large.  
	Let the Hensel expansion of $\zeta$ be
	\[
	\zeta= \sum_{i=0}^{\infty} a_{i}p^{i},
	\quad a_{i}\in \{0,1,\ldots,p-1\}. 
	\]
	For $j \ge 1$, set $\sigma_{j}= \lfloor\log x_{j,0}/\log p\rfloor$ so that 
	$$
	x_{j,0}\asymp |x_{j,1}| \asymp p^{\sigma_{j}}. 
	$$
	Then the second claim of Theorem~\ref{hopefully} implies	that $\sigma_{j+1}/\sigma_{j}$ tends to infinity
	with $j$. Partition the integers greater than or equal to $\sigma_1$ 
	into the intervals $I_{j}:=[\sigma_{j}, \sigma_{j+1})\cap \mathbb{Z}$.
	 		
	We construct $\xi$ with the desired properties
	by manipulating the Hensel expansion of $\zeta$. First,
	we derive from the sequence $(\sigma_{j})_{j\geq 1}$ two 
	other positive integer sequences $(\tau_{j})_{j\geq 1},
	(\nu_{j})_{j\geq 1}$ defined by 
		\begin{equation}  \label{eq:ss1}
	\nu_{j}=\lfloor t \mu \sigma_{j}\rfloor+C, \quad \tau_{j}=\lfloor \mu \nu_{j}\rfloor, 
	\end{equation}
	for some large positive integer constant $C$. 
		For $x_{1,0}$ and $|x_{1,1}|$ sufficiently large, we have 
	\[
	\sigma_{1}<\nu_{1}<\tau_{1}<\sigma_{2}<\nu_{2}<\tau_{2}<\cdots,
	\]
	and since the quotient $\sigma_{j+1} / \sigma_j$ tends to infinity with $j$ we also have
	\begin{equation}  \label{eq:ss2}
	\lim_{j\to\infty} \frac{\sigma_{j+1}}{\tau_{j}}= \infty. 
	\end{equation}
	Let
	\[
	J_{j}=\{\nu_{j},\nu_{j}+1,\ldots,\tau_{j}\}=[\nu_{j}, \mu  \nu_{j}]\cap \mathbb{Z},
	\quad j\geq 1,
	\]
	so that $J_{j}\subseteq I_{j}$, for $j$ sufficiently large. 
	Consider the 
	$p$-adic number
	\[
	\xi = \sum_{i=0}^{\infty} b_{i}p^{i},
	\quad b_{i}\in \{0,1,\ldots,p-1\},
	\]
	derived from $\zeta$ by setting 
	\[
	b_i= \begin{cases} 0, \quad i\in \cup_{j} (J_{j}\setminus \{ \nu_j , \tau_{j}\}), \\
	               1, \quad i\in \cup_{j} \{ \nu_j , \tau_{j}\}, \\
	           a_{i}, \quad i \notin \cup_{j} J_{j}.         \end{cases} 
	\]
	In other words the Hensel expansions of $\xi$ and $\zeta$ coincide
	outside the intervals $J_{j}$, whereas the digits of $\xi$ are all zero 
	inside $J_{j}$, except at the first and last position of any $J_{j}$, where
	for technical reasons we put the digit $1$.
		We will show that
	\begin{equation}  \label{eq:total}
	\mu^{\astmult}(\xi)= 2\mu, \quad \mu(\xi)=\tilde{\mu}= t \mu, 
	\end{equation}
	if $\mu$ is sufficiently large. This will prove the theorem as $t$ is arbitrary in $[1, 2]$.

	We start with the easiest of the four inequalities, namely
	\begin{equation}  \label{eq:abs1}
	\mu^{\astmult}(\xi)\geq 2\mu. 
	\end{equation}
	Define the integers
	\begin{equation}  \label{eq:njdef}
	N_{j}=   \sum_{i=0}^{\nu_{j}}
	b_{i}p^{i}, \quad j\geq 1.
	\end{equation}
	Clearly $N_j \ll p^{\nu_j}$. Morever, as $\xi$ has
	zero digits at places ranging from
	$\nu_j+1$ to $\tau_j-1\approx \mu \nu_j$, the integers $N_j$ approximate 
	$\xi$ at the order roughly $\mu$, hence
	\begin{equation}  \label{eq:01}
	|\xi-N_{j}|_{p}\ll p^{-\tau_{j}} \ll p^{-\nu_{j} \mu }\ll N_{j}^{-\mu}
	=(\sqrt{1\cdot N_{j}})^{-2\mu}.
	\end{equation}
    We directly deduce \eqref{eq:abs1} from \eqref{eq:01}. We remark that
    $|\xi-N_{j}|_{p}\geq p^{-\tau_{j}}$ since $b_{\tau_{j}}=1$,
    and furthermore, since $b_{\nu_j}=1$, we get in fact
    $N_j \asymp p^{\nu_j}$.
    So we can refine \eqref{eq:01} as
    \begin{equation}  \label{eq:follow}
    |\xi-N_{j}|_{p} \asymp p^{-\nu_j \mu}\asymp N_j^{-\mu}.
    \end{equation}

	Next we show that 
	\begin{equation} \label{eq:h}
	\mu(\xi)\geq t \mu.
	\end{equation}
	By \eqref{eq:follow} the pairs $(x_{0},x_{1})=(N_{j}, 1)$
	only induce approximations of quality $\mu$.
	By manipulating the pairs
	$(x_{j,0}, x_{j,1})$ associated to $\zeta$,
	we construct better approximating sequences $(y_{j,0})_{j\geq 1}, (y_{j,1})_{j\geq 1}$
	such that,
	 for any given $\varepsilon_1>0$ and sufficiently large $j\geq j_{0}(\varepsilon_1)$, we have 
	\begin{equation} \label{eq:thus1}
	|y_{j,1}\xi-y_{j,0}|_{p} \ll \max\{ |y_{j,0}|, |y_{j,1}| \}^{-t \mu+\varepsilon_1}. 
	\end{equation}
	This obviously implies \eqref{eq:h}.
	To construct suitable $y_{j,0},y_{j,1}$, 
	recall that 
	$|x_{j,0}|\asymp |x_{j,1}|\asymp p^{\sigma_{j}}$ 
	and $\sigma_{j}<\nu_{j}<\tau_{j}<\sigma_{j+1}$ for $j \ge 1$, 
	with 
	$$
	\lim_{j \to \infty} \frac{\nu_{j}}{\sigma_{j}} = t \mu, \quad 
	\lim_{j \to \infty} \frac{\sigma_{j+1}}{\tau_{j}} = + \infty, \quad
	\lim_{j \to \infty} \frac{\tau_{j}}{\nu_{j}} = \mu.
	$$
	For $i\geq 1$ define
	$$ 
	u_{i}=  \sum_{j\in J_{i} } a_{j}p^{j} -p^{\nu_{i}} -p^{\tau_{i}}, \qquad u^{(i)}=u_{1}+u_{2}+\cdots+u_{i}.
	$$ 
	Notice that by construction $\zeta-\xi$ is the infinite sum $u_1 + u_2 + \ldots $.

	 Moreover, assuming that $\epsilon < (\mu - 2)/2$, we note that
	\begin{equation}  \label{eq:uinorm}
	 p^{\tau_i(\frac{1}{2} - \epsilon)}\ll |u^{(i)}|\ll p^{\tau_i}, 
	 \qquad\qquad i\geq 1.
	\end{equation}
	The right estimate is obvious. If the left one is not satisfied, then
	$a_j=0$ for $\lfloor\eta_i\rfloor\leq j\leq \tau_i-1$, 
	where $\eta_i:= (\frac{1}{2}-\epsilon)\tau_i$ and $a_{\tau_{i}}=1$. But then
	the integer $M_i= \sum_{j\leq \lfloor \eta_i\rfloor } a_jp^{j}$ satisfies
	\[
	| \zeta- M_i|_p \ll p^{-\tau_i}\ll M_i^{- \tau_i / \eta_i }\ll
	M_i^{- 1 / (\frac{1}{2}-\epsilon) }= M_{i}^{-2-\epsilon- \frac{3\epsilon+2\epsilon^2}{1-2\epsilon} }, 
	\]
	a contradiction with Theorem~\ref{hopefully} for large $i$.

	We claim that if we set
	\begin{equation}  \label{eq:hh}
	y_{j,0}= x_{j,0}- u^{(j-1)}x_{j,1}, \quad
	y_{j,1}= x_{j,1}, 
	\end{equation}
	then indeed \eqref{eq:thus1} holds. 
	We rearrange
	\begin{align}  \label{eq:z1}
	|y_{j,1}\xi-y_{j,0}|_{p}&=
	|x_{j,1}(\xi+u^{(j-1)})-x_{j,0}|_{p}
	\\&=|x_{j,1}(\xi+u^{(j-1)}-\zeta)+(x_{j,1}\zeta-x_{j,0})|_{p}. \nonumber\\
	&\leq \max\{ |x_{j,1}(\xi+u^{(j-1)}-\zeta)|_p, |x_{j,1}\zeta-x_{j,0}|_{p}  \}.  \nonumber
	\end{align}
	By assumption the latter term satisfies
	\begin{equation}  \label{eq:latters}
	|x_{j,1}\zeta-x_{j,0}|_{p} \ll x_{j,0}^{-t \mu}.
	\end{equation}
	To estimate the former expression, note that by construction
	the Hensel expansions of $\xi+u^{(j-1)}$
	and $\zeta$ coincide up to digit $\nu_{j}-1$ (last digit
	before the interval $J_{j}$ starts). 
	Thus, we have 
	\begin{equation}  \label{eq:z2}
	|x_{j,1}(\xi+u^{(j-1)}-\zeta)|_{p}\leq |\xi+u^{(j-1)}-\zeta|_{p} \ll p^{-\nu_{j}}\ll x_{j,0}^{-\nu_{j}/\sigma_{j}}\ll x_{j,0}^{-t \mu},
	\end{equation}
	where the last estimate follows from \eqref{eq:ss1}.
    By combining 
	\eqref{eq:z1}, \eqref{eq:latters}, and \eqref{eq:z2}, we derive
	\begin{equation}  \label{eq:fr}
	|y_{j,1}\xi-y_{j,0}|_{p} \ll x_{j,0}^{-t \mu}\ll |x_{j,1}|^{-t \mu }= |y_{j,1}|^{-t \mu}.
	\end{equation}
	Now for given $\varepsilon_2>0$ and large $j\geq j_{0}(\varepsilon_2)$,  
	we get from \eqref{eq:ss2} the estimate
	\[
	|u^{(j-1)}|\ll p^{\tau_{j-1}}< p^{\varepsilon_2 \sigma_{j}}\ll x_{j,0}^{\varepsilon_2}.
	\]
	Combined with \eqref{eq:fr} and recalling that $x_{j,0}\asymp |x_{j,1}|$ we infer
	\begin{equation}  \label{eq:west}
	|y_{j,0}|=|x_{j,0}-u^{(j-1)}x_{j,1}| \ll x_{j,0}+|u^{(j-1)}|\cdot |x_{j,1}|\ll
	|x_{j,1}|^{1+\varepsilon_2}=|y_{j,1}|^{1+\varepsilon_2},
	\end{equation}
	hence we derive \eqref{eq:thus1} from \eqref{eq:fr}, 
	and consequently \eqref{eq:h} follows. At this point we notice that
	the reverse inequality $|y_{j,0}|\gg |y_{j,1}|$
	follows similarly via
	\begin{equation}  \label{eq:jordan}
	|y_{j,0}|=|x_{j,0}-u^{(j-1)}x_{j,1}|\gg |x_{j,1}|\cdot |u^{(j-1)}|\geq |x_{j,1}|=|y_{j,1}|,
	\end{equation}
	where we use that $x_{j,0}\asymp |x_{j,1}|$ and the fact 
	that $|u^{(j)}|$ tends to infinity with $j$ by \eqref{eq:uinorm}.
	So we keep in mind for the sequel that
	all the integers $x_{j,0}, |x_{j,1}|, |y_{j,0}|, |y_{j,1}|$ are
	of comparable size, in the sense that, for every 
	$\delta > 0$ and for every sufficiently large $j$, we have
	$$
	\max\{ x_{j,0}, |x_{j,1}|, |y_{j,0}|, |y_{j,1}| \} \le (\min \{ x_{j,0}, |x_{j,1}|, |y_{j,0}|, |y_{j,1}| \} )^{1 + \delta}. 
	$$
	
	Next we show the reverse estimate
	\begin{equation}  \label{eq:abs3}
	\mu(\xi)\leq t \mu.
	\end{equation}
	Assume otherwise that there are integers $x,y$ with $\max\{ |x|,|y|\}$
	arbitrarily large 
	and $\theta> t \mu$ such that 
	\begin{equation}  \label{eq:contradict}
	|y \xi-x|_{p} \le \max\{ |x|,|y|\}^{-\theta}.
	\end{equation} 
	We may assume that $x$ and $y$ are coprime, by the comments in Section~\ref{intro}.
	We distinguish two cases. 
	
	\underline{Case 1}: The pair $(x,y)$ is among the pairs $(y_{j,0}, y_{j,1})$
	defined in \eqref{eq:hh} above. We show the reverse
	estimate to \eqref{eq:thus1}, that is, 
	\begin{equation}  \label{eq:sueden}
	|y_{j,1}\xi-y_{j,0}|_{p} \gg \max\{ |y_{j,0}|, |y_{j,1}| \}^{-t \mu}, \quad j\geq 1.
	\end{equation}
	This clearly contradicts \eqref{eq:contradict} for these pairs. 
	By assumption the reverse estimate to \eqref{eq:latters}
	holds as well, i.e.
	\[
	|x_{j,1}\zeta-x_{j,0}|_{p} \gg x_{j,0}^{-t \mu}.
	\]
	Recall that  for $a,b$ in $\mathbb{Q}_{p}$ if $|a|_{p}\neq |b|_{p}$ then $|a+b|_{p}=\max\{ |a|_{p},|b|_{p}\}$.
	Now by \eqref{eq:ss1} the upper bound in \eqref{eq:z2} 
	with the parameter $C$ taken large enough
	will be strictly smaller than this value $x_{j,0}^{-t \mu}$, 
	hence applying the above argument to 
	\[
	a= x_{j,1}\zeta-x_{j,0} , \quad b= x_{j,1}(\xi+u^{(j-1)}-\zeta)
	\]
	allows us, in view of \eqref{eq:z1}, to derive
	\[
	|y_{j,1}\xi-y_{j,0}|_{p} \gg x_{j,0}^{-t \mu} \gg |x_{j,1}|^{-t \mu}=|y_{j,1}|^{-t \mu}\geq \max\{ |y_{j,0}|, |y_{j,1}| \}^{-t \mu},
	\]
	our desired lower bound \eqref{eq:sueden}. 
	
	\underline{Case 2}: The pair $(x,y)$ is not among 
	the $(y_{j,0}, y_{j,1})$. 
	Write
	\[
	H= \max\{ |x|, |y|\}.
	\]
	In fact we show that then
	\begin{equation}  \label{eq:osten}
	|y \xi- x|_{p} \gg H^{-\mu -\varepsilon_3}.
	\end{equation}
	Since $\mu\leq t \mu$ this clearly implies \eqref{eq:abs3}. Note that the bound is optimal
	as by \eqref{eq:follow} it is attained 
	with $\varepsilon_3=0$ by $(x,y)=(N_j, 1)$. However, by the same
	argument, we can exclude these pairs in 
	our investigation.   
	For other pairs, 
	we verify \eqref{eq:osten} indirectly by showing 
	that any pair $(x,y)$ that violates
	the inequality induces a reasonably good 
	rational approximation to $\zeta$ which is not
	among the $x_{j,0}/x_{j,1}$, contradicting the third claim 
	of Theorem~\ref{hopefully}. So, assume that for some $(x,y)$ as above
	we have
	\begin{equation}  \label{eq:westen}
	|y \xi- x|_{p} \leq H^{-\mu -\varepsilon_3}.
	\end{equation}
	Below \eqref{eq:jordan} 
	we noticed that $|y_{j,1}|=|x_{j,1}|$ and $|y_{j,0}|$ are of 
	comparable size, all being 
	roughly equal to $x_{j,0}\asymp |x_{j,1}|$. In particular, the sequences 
	$(|y_{j,0}|)_{j \ge 1}$ and $(|y_{j,1}|)_{j \ge 1}$ are increasing. 
	For a pair $(x,y)$ satisfying \eqref{eq:westen},
	let $h$ be the index with
	\[
	\max\{ |y_{h,0}|, |y_{h,1}| \}< H \leq \max\{ |y_{h+1,0}|, |y_{h+1,1}| \}.
	\]
	In view of \eqref{eq:thus1} and \eqref{eq:westen},
	we derive from Corollary~\ref{korollar}
	\begin{equation} \label{eq:3exp}
	\max\{ |y_{h+1,0}|, |y_{h+1,1}|\}^{\frac{1}{\mu - 1}+\varepsilon_4}
	\gg H \gg \max\{ |y_{h,0}|, |y_{h,1}|\}^{t \mu - 1 -\varepsilon_4}.
	\end{equation}
	However, the right hand estimate 
	is not sufficient. We show the stronger lower bound 
	\begin{equation}  \label{eq:weshow}
	H \gg \max\{ |y_{h,0}|, |y_{h,1}|\}^{t \mu (\mu - 1)}, 
	\end{equation}
	again by application of Lemma~\ref{padicle}. For simplicity write
    \[
    s= \frac{\log H}{\log x_{h,0}}.
    \]
    Recalling
	that all $x_{j,0}, x_{j,1}, y_{j,0}, y_{j,1}$ are in absolute value roughly
	of the same size,
	we have to show $s\geq t \mu (\mu - 1) -\varepsilon_{5}$ for arbitrarily
	small $\varepsilon_{5}>0$.
	According to \eqref{eq:3exp}, upon increasing $\varepsilon_{4}$
	to take into account the implied constants if necessary, 
	we can assume $s\geq t \mu - 1 -\varepsilon_{6}$ for arbitrarily
	small $\varepsilon_{6}>0$.
	On the one hand, with $N_{j}$ as in \eqref{eq:njdef} 
	in view of \eqref{eq:01} we have
	\[
	\max\{ |y \xi- x|_{p} , |\xi-N_{h}|_{p} \} \ll
	\max\{ H^{-\mu -\varepsilon_3} , p^{-\nu_h \mu } \} 
	\]
	and, since 
	\[
	H= x_{h,0}^{s},\quad p^{\nu_{h}}	\asymp x_{h,0}^{\nu_{h}/\sigma_{h}}\asymp 
	x_{h,0}^{t \mu}, 
	\]
	we get 
	\[
	\max\{ |y \xi- x|_{p} , |\xi-N_{h}|_{p} \} \ll 
	x_{h,0}^{- \mu  \min\{t \mu,s \}  }.
	\]
	On the other hand, as
	${\rm gcd}(x,y)=1$ and ${\rm gcd}(N_j,1)=1$  and
	we have assumed $(x,y)\neq (N_{j},1)$, 
	these pairs are linearly independent. 
	Hence, from Lemma~\ref{padicle} 
	and $N_h\ll p^{\nu_{h}}\ll x_{h,0}^{t \mu}$, we get
	\[
	\max\{ |y \xi- x|_{p} , |\xi-N_{h}|_{p} \} \gg
	H^{-1}N_{h}^{-1}\gg x_{h,0}^{-s- t \mu}.
	\]
	This gives the lower bound
	\begin{equation}  \label{eq:bl}
	s+ t \mu +\varepsilon_{7} \geq \mu \min\{ s, t \mu\}.
	\end{equation}
	If $s \le t \mu$, then 
	we get $s\leq  t \mu/ (\mu - 1) + \varepsilon_{8}$. However, 
	in view of $s\geq t \mu - 1 -\varepsilon_{6}$ noticed above as $\varepsilon_{6}, \varepsilon_{8}$ can be arbitrarily small this
	gives a contradiction 
	as soon as 
	\[
	\mu> 1 + \frac{t \mu}{t \mu - 1} = 2+ \frac{1}{t\mu-1}.
	\]
	Since $t \mu \geq \mu >2$ a sufficient criterion is 
	\begin{equation} \label{eq:str1}
	\mu > 3.
	\end{equation}
	If $s > t \mu$, then  we 
	derive from \eqref{eq:bl} that $s+ t \mu+\varepsilon_{7} \geq t \mu^2$ or
	equivalently $s\geq t \mu ( \mu - 1)-\varepsilon_{7}$. Thus,
	as $\varepsilon_{7}$ can be arbitrarily small,
	we have shown \eqref{eq:weshow}.
	
	Next observe that the triangle inequality gives
	\[
	|x + yu^{(h)}-y\zeta|_{p} = |(x-y\xi)+y(\xi+u^{(h)}-\zeta)|_{p}\leq
	\max\{ |x-y\xi|_{p} , |y(\xi+u^{(h)}-\zeta)|_{p}\},
	\]
	so combined with \eqref{eq:z2} applied for $j=h+1$
	and with \eqref{eq:westen} we conclude
	\begin{equation}  \label{eq:2exp}
	|x+y u^{(h)}- y \zeta|_{p} \ll \max\{\; x_{h+1,0}^{- t \mu}\;,\;  H^{-\mu -\varepsilon_{3} }\; \}.
	\end{equation}
	From \eqref{eq:3exp} and as $|y_{h,0}|\ll |y_{h,1}|^{1+\varepsilon_2}=|x_{h,1}|^{1+\varepsilon_2}$ by \eqref{eq:west}
	and $\mu>2$ and $t\geq 1$,
	we check that 
	\[
	H \ll 
	x_{h+1,0}^{ \frac{1}{\mu - 1} +\varepsilon_{9} } \ll
	x_{h+1,0}^{t},
	\]
	so the right hand expression in \eqref{eq:2exp} is larger. 
	With \eqref{eq:3exp} and \eqref{eq:weshow}, 
	and since $|y_{j,1}|$, $|y_{j,0}|$ and $x_{j,0}$
	are of comparable size and $\tau_{h}/\sigma_{h}$ tends to $(t\mu)\mu= t \mu^2$
	by \eqref{eq:ss1}, we estimate
	\begin{align*} 
	\max\{ |y|, |x+yu^{(h)}|  \} &\leq |x|+|y|\cdot |u^{(h)}|\leq
	H+H \cdot  p^{\tau_{h}+1}  \\ &\ll 
	 H\cdot x_{h,0}^{\tau_{h}/\sigma_{h}}\ll
	H\cdot (H^{\frac{1}{(\mu-1)t\mu}})^{\tau_{h}/\sigma_{h}}
	\ll  H^{2+ \frac{1}{\mu - 1} }.
	\end{align*}
	To sum up, we have shown 
\begin{align}  \label{eq:dotter}
\max\{ |y|, |x+yu^{(h)}|  \} &\ll H^{2+ \frac{1}{\mu - 1} }, \\
|x+yu^{(h)}-y \zeta|_{p} &\leq H^{-\mu -\varepsilon_{3} }. \nonumber
\end{align}
	From \eqref{eq:dotter}, we get
	\begin{align*}
	-\frac{\log |x+ yu^{(h)}-y \zeta|_{p}}{\log \max\{ |y|, |x+yu^{(h)}|  \}}
	& \geq (\mu+\varepsilon_{3})\cdot \frac{\log H}{\log \max\{ |y|, |x+yu^{(h)}|\}}\\
	& \geq (\mu+\varepsilon_{3})\cdot \left(2+\frac{1}{\mu - 1}\right)^{-1} - \varepsilon_{10}
	> \frac{\mu^{2} - \mu}{2\mu - 1} - \varepsilon_{10}.
	\end{align*}
	Thus we have found integers $z_{0},z_{1}$ with
	\[
	-\frac{\log |z_{1}\zeta-z_{0}|_{p}}{\log \max\{|z_{0}|,|z_{1}|\}}\geq \frac{\mu^{2} - \mu}{2\mu -1} - \varepsilon_{10}.
	\]
	Thereby, as $\varepsilon_{10}$ can be arbitrarily small, if
	\begin{equation}  \label{eq:str2}
	\mu>\frac{5+\sqrt{17}}{2},
	\end{equation}
	then $\frac{\mu^{2} - \mu}{2\mu -1} > 2$ and 
	we have constructed an approximation of order greater than $2$
	to $\zeta$. If $\epsilon>0$ from 
	Theorem~\ref{hopefully} for our $\zeta$ 
	has been chosen small enough (depending on the given $\mu$), concretely for
	\[
	\epsilon= \frac{1}{2} \biggl( \frac{\mu^{2} - \mu}{2\mu -1} - 2  \biggr),
	\]
	by the assumptions of the theorem and since 
	 ${\rm gcd}(x,y)=1$
	and ${\rm gcd}(x_{j,0},x_{j,1}) = 1$, 
	this implies 
	$(x+yu^{(h)},y)= (x_{j,0},x_{j,1})$ for some $j$. 
	We assume this is the case and will derive a contradiction. 
	
	We first show that $j$ cannot exceed $h$.  
	Note that the case $j=h+1$ has already been treated in Case 1. 
	Since $x_{j,0}\asymp |x_{j,1}|$ and $|u^{(j)}|$ tends to infinity with $j$
	by \eqref{eq:uinorm}, 
	we must have
	\[
	|x|\asymp |y|\cdot |u^{(h)}|\asymp x_{j,0}\cdot p^{\tau_h}\asymp x_{j,0}\cdot x_{h,0}^{\tau_h /\sigma_h }\asymp x_{j,0}\cdot x_{h,0}^{t\mu^2 }.
	\]
	If $j\geq h+2$ then this clearly contradicts $|x|\leq H\ll x_{h+1,0}$
	from \eqref{eq:3exp}. 
	Now, we assume that $j \le h$. 
	It follows from \eqref{eq:z2} that 
	\begin{align*}
	|\xi y-x|_{p} &= | \xi x_{j,1} - x_{j,0} + x_{j,1}u^{(h)}|_{p} \\
	&= | (u^{(h)}+\xi-\zeta) x_{j,1} + (\zeta x_{j,1}- x_{j,0})|_{p} \\
	&\geq |\zeta x_{j,1}- x_{j,0}|_{p} - |(u^{(h)}+\xi-\zeta) x_{j,1}|_{p}\\
	&\gg x_{j,0}^{-t\mu} - x_{h+1,0}^{-t\mu}.
	\end{align*}
	Now the crude estimate $x_{h+1,0}\gg  x_{h,0}^{ \sigma_{h+1}/\sigma_{h} }\gg x_{h,0}^{ \tau_{h}/\sigma_{h} }
	\gg x_{h,0}^{t\mu^2}\gg x_{h,0}^{4}\gg x_{j,0}^{4}$ suffices to derive
	\[
	|\xi y-x|_{p} \gg x_{j,0}^{-t\mu}.
	\]
	But on the other hand by assumption \eqref{eq:westen} and
	\eqref{eq:weshow} we get
	\[
	|\xi y-x|_{p} \ll H^{-\mu}\ll x_{j,0}^{-t\mu^2(\mu-1)}.
	\]
	The combination of the latter inequalities gives the desired contradiction.

  We see that $\mu>(5+\sqrt{17})/2$
	is the most restrictive one among the conditions 
	\eqref{eq:str1}, \eqref{eq:str2}  
	we have collected on the way, which imposes $\mu^{\times}=2\mu >  5+\sqrt{17}$, that is, 
	the restriction made in the theorem.

	Finally we show  
	\begin{equation}  \label{eq:spezialfall}
	\mu^{\astmult}(\xi) \leq 2\mu.
	\end{equation}
	We again distinguish between rationals $y_{j,0}/y_{j,1}$ and
	other rationals. Concerning the first type, 
	again because $|y_{j,0}|$ and $|y_{j,1}|$ are of comparable size
	and from \eqref{eq:sueden} we indeed derive that 
	\[
	|y_{j,1}\xi-y_{j,0}|_{p} \gg \left(\sqrt{|y_{j,0} \cdot y_{j,1}|}\right)^{- t \mu -\varepsilon_{11} }.
	\]
	Thus the exponent restricted to this case satisfies 
	$\mu^{\astmult}(\xi)\leq t \mu+\varepsilon_{11}\leq 2\mu +\varepsilon_{11}$.
	Since $\varepsilon_{11}$ can be taken arbitrarily small the claim follows.
	Finally in the latter case where $(x,y)\neq (y_{j,0},y_{j,1})$
	we conclude with \eqref{eq:osten} via 
	\[
	|y \xi-x|_{p} \gg \max\{ |x|,|y| \}^{- \mu -\varepsilon_3}\geq
	\left(\sqrt{|xy|} \right)^{-2\mu -2\varepsilon_3},
	\]
	again giving $\mu^{\astmult}(\xi)\leq 2\mu +2\varepsilon_3$ 
	and the claim \eqref{eq:spezialfall} follows as $\varepsilon_3$
	can be taken arbitrarily small.
	The proof of \eqref{eq:total} and thus of the theorem is complete.
	\end{proof}

\section{Proofs of Theorem~\ref{endlich} and Theorem~\ref{neu} }  \label{se8}

\begin{proof}[Proof of the first assertion \eqref{eq:andere} 
	of Theorem~\ref{endlich}]
Assume that for some $\mu\geq 2$, there exist non-zero coprime 
integers $x,y$ with $|x y|$ arbitrarily large and such that 
\begin{equation}  \label{eq:hha}
|y \xi- x|_{p} = (Q^{\astmult})^{ -\mu}, \quad Q^{\astmult} =\sqrt{|xy|}.
\end{equation}
Set $Q=\max\{ |x|,|y|\}$. Define $A$ in $[1,2]$ and $\tau$ by 
\[
Q=(Q^{\astmult})^{ A}, \quad \tau= \frac{\mu -A}{2}.
\]
In the sequel, $\varepsilon_1, \varepsilon_2, \ldots$ denote
positive real numbers that can be taken arbitrarily small as $Q^{\times}$ tends to infinity. 
Set
\[
X= (Q^{\astmult})^{ \tau-\varepsilon_1}.
\]
By \eqref{eq:hha}, for any positive integer $M$ with 
$\sqrt{| Mx \cdot My | }=MQ^{\astmult}\leq X$, we have
\begin{align}  
|My \xi-Mx|_{p} \ge M^{-1}(Q^{\astmult})^{ -\mu } & \geq  \frac{Q^{\astmult}}{X }(Q^{\astmult})^{ -\mu } 
\nonumber  \\
&\geq X^{-\frac{\mu - 1}{\tau}-1 - \varepsilon_{2} }= X^{-\frac{2(\mu - 1)}{\mu - A}-1 - \varepsilon_{2} }.  
 \label{eq:othcase0}
\end{align}
Consider an integer 
pair $(\tilde{x},\tilde{y})$ linearly independent to $(x,y)$
and with $\sqrt{|\tilde{x}\tilde{y}|}\leq X$. 
Set $\tilde{X}:=\max\{ |\tilde{x}|,|\tilde{y}| \}$ and observe that $\tilde{X} \leq X^{2}$.
By construction, we have 
\[
|y \xi - x|_{p}  = 
(Q^{\astmult})^{ -\mu } = Q^{-\frac{\mu}{A} }
\le Q^{-1  }X^{-2} \le Q^{-1  }\tilde{X}^{-1},
\]
where we have used that 
\[
X^{2} = (Q^{\astmult})^{\mu -A-2 \varepsilon_1} =
Q^{\frac{\mu -A}{A} -\varepsilon_3 }.
\]
It then follows from Lemma~\ref{padicle} that 
\begin{align}   
|\tilde{y}\xi - \tilde{x}|_{p} \gg Q^{-1} \tilde{X}^{-1}\gg Q^{-1}X^{-2}
& = X^{-2- \frac{A}{\tau-\varepsilon_1} } \nonumber  \\
& = X^{-2- \frac{A}{\tau} -\varepsilon_{4}}=
 X^{- \frac{2 \mu}{\mu -A} -\varepsilon_{4}}.   \label{eq:vergleich}     
\end{align}
Since $\mu$ can be chosen arbitrarily close to $\mu^{\astmult}(\xi)$, we deduce from 
\eqref{eq:othcase0} and \eqref{eq:vergleich} that 
\begin{equation}  \label{eq:hat}
\hmu^{\astmult}(\xi) \leq  \sup_{A\in[1,2] } 
\max \left\{ \frac{3\mu^{\astmult}(\xi) - 2 - A}{\mu^{\astmult}(\xi)-A} 
\; , \;   \frac{2 \mu^{\astmult}(\xi) }{\mu^{\astmult}(\xi) -A} \right\}.       
\end{equation}
For $\mu^{\astmult}(\xi)\geq 4$ it is readily checked that, for any $A$ in $[1,2]$, the quantity on the
left is greater than or equal to the quantity on the right. 
Hence, for $\mu^{\astmult}(\xi)\geq 4$, we 
have proved that 
\begin{equation}  \label{eq:hut}
\hmu^{\astmult}(\xi) \leq  \sup_{A\in[1,2] }  \frac{3\mu^{\astmult}(\xi) - 2 - A}{\mu^{\astmult}(\xi)-A}  
= 3+\frac{2}{\mu^{\astmult}(\xi)-2}. 
\end{equation}
Since \eqref{eq:hut} clearly holds if $\mu^{\astmult}(\xi) < 4$ (recall that $\hmu^{\astmult}(\xi) \leq 4$ is    
always true), this proves   
the first claim \eqref{eq:andere} of the theorem. 
\end{proof}

Actually, in the preceding proof, we have shown a slightly stronger result than \eqref{eq:hat}, 
which we state in the following corollary for later use. 

\begin{corollary}  \label{KOR}
Let $\xi$ be an irrational $p$-adic number. 
Assume that there exist $A$ in $[1, 2]$
and an infinite sequence $\cS$ of pairs of nonzero integers $(x, y)$ such that
$$
\limsup_{(x,y) \in \cS, \max\{|x|, |y|\} \to \infty} \, \frac{- \log |y \xi- x|_{p}}{ \log \sqrt{|xy|}} = \mu^{\astmult}(\xi)
$$
and
$$
\lim_{(x,y) \in \cS, \max\{|x|, |y|\} \to \infty} \, \frac{ \log \max\{|x|, |y|\}}{ \log \sqrt{|xy|}} = A. 
$$
Then, we have 
\begin{equation}  \label{eq:hatbis}
\hmu^{\astmult}(\xi) \leq   \max \left\{ \frac{3\mu^{\astmult}(\xi) - 2 - A}{\mu^{\astmult}(\xi)-A} 
\; , \;   \frac{2 \mu^{\astmult}(\xi) }{\mu^{\astmult}(\xi) -A} \right\}. 
\end{equation}
In particular, if we have $A=1$, that is, if $|x|$ and $|y|$ are of comparable size for every pair 
$(x, y)$ in $\cS$, then we obtain $\hmu^{\astmult}(\xi) \leq 3$.
\end{corollary}

\begin{proof}
	The estimate \eqref{eq:hatbis} comes directly from the proof
	of Theorem~\ref{endlich} above.
	For the last assertion, observe that when $A=1$ the left hand side of \eqref{eq:hatbis} is equal to $3$, 
	while the right hand side is at most equal to $3$ when $\mu^{\astmult}(\xi)\geq 3$.
\end{proof}

For a given $p$-adic number $\xi$,
we define the sequence of multiplicative best approximation pairs 
$((x_{k}^{\astmult},y_{k}^{\astmult}))_{k \ge 1}$
in a similar way as for the usual approximation problem, by 
looking at the pair of {\em coprime} integers $(x,y)$ minimizing
$|y \xi- x|_{p}$ among all the 
integer pairs with $0<\sqrt{|xy|}\leq X$, and letting the positive real number 
$X$ grow to infinity. 
Write $Q_{k}^{\astmult}= \sqrt{|x_{k}^{\astmult} y_{k}^{\astmult}|}$ for $k \ge 1$.
By construction, we have
$$
Q_{1}^{\astmult}<Q_{2}^{\astmult}<\cdots, \quad |y_{1}^{\astmult}\xi - x_{1}^{\astmult}|_{p} 
> |y_{2}^{\astmult}\xi - x_{2}^{\astmult}|_{p} > \cdots.
$$
Furthermore, $|y_{k+1}^{\astmult}\xi - x_{k+1}^{\astmult}|_{p}$ is smaller than
$|My_{k}^{\astmult}\xi - Mx_{k}^{\astmult}|_{p}$ as soon as the 
positive integer $M$ satisfies $M<Q_{k+1}^{\astmult}/Q_{k}^{\astmult}$.
Observe that, by the remark on the coprimality of $x$ and $y$ following Definition \ref{def1}, we have 
$$
\mu^{\astmult}(\xi) = \limsup_{k \to \infty} 
\, \frac{- \log |y_{k}^{\astmult}\xi - y_{k}^{\astmult}|_{p}}{\log Q_k^{\astmult}} 
$$
and
$$ 
\widehat{\mu}^{\astmult}(\xi) 
= 1 + \liminf_{k \to \infty} \, \frac{- \log |y_{k}^{\astmult}\xi - y_{k}^{\astmult}|_{p} -
\log Q_{k}^{\astmult} }{\log Q_{k+1}^{\astmult}}. 
$$

We begin with an auxiliary result on the sequence of best approximations.

\begin{lemma}  \label{hilfl}
	With the above notation, we have
	\[
	\limsup_{k \to \infty} \, \frac{\log Q_{k+1}^{\astmult}}{\log Q_{k}^{\astmult} } 
	\leq \frac{\mu^{\astmult}(\xi) - 1}{\hmu^{\astmult}(\xi) - 1}.
	\]
\end{lemma}

\begin{proof}
By the definitions of the limsup and of the liminf, we get 
that, for every $\varepsilon > 0$ and every large $k$, we have
$$
(\mu^{\astmult}(\xi) + \varepsilon) \log Q_k^{\astmult} \ge - \log |y_{k}^{\astmult}\xi - x_{k}^{\astmult}|_{p}  
$$
and
$$
(\widehat{\mu}^{\astmult}(\xi) - 1 -  \varepsilon)  \log Q_{k+1}^{\astmult} 
\le - \log |y_{k}^{\astmult}\xi - x_{k}^{\astmult}|_{p}  - \log Q_{k}^{\astmult}. 
$$
This gives
$$
(\widehat{\mu}^{\astmult}(\xi) - 1 -  \varepsilon)  \log Q_{k+1}^{\astmult} \le 
(\mu^{\astmult}(\xi) - 1 + \varepsilon) \log Q_k^{\astmult},
$$
and the lemma follows. 
\end{proof}

\begin{proof}[Proof of Theorem~\ref{neu}]

	We establish (i) and observe that (ii) can be proved analogously.
	
	Assume that we have $|x_{k}^{\astmult} | \gg |y_{k}^{\astmult}|$ 
	and $|x_{k+1}^{\astmult}| \gg |y_{k+1}^{\astmult}|$.	
	Recall that $Q_{k}^{\astmult}=\sqrt{ |x_{k}^{\astmult} y_{k}^{\astmult} |}$. 
	Define $\alpha_k$ and $\beta_k$ by 
	 \[
	 (Q_{k}^{\astmult})^{ \alpha_{k}}= |x_{k}^{\astmult}|,\quad (Q_{k}^{\astmult})^{ \beta_{k}}= |y_{k}^{\astmult}|,
	 \]
	 and note that $\alpha_{k}+\beta_{k}=2$.
	Define $\mu_k$ (it would be more appropriate to write $\mu_k^\astmult$, but for the sake of readability
	we choose to drop the ${}^\astmult$) by 
	\[
	|y_{k+1}^{\astmult} \xi - x_{k+1}^{\astmult}|_{p} < |y_{k}^{\astmult} \xi - x_{k}^{\astmult}|_{p} 
	= (Q_{k}^{\astmult})^{-\mu_k}.
	\]
	Now, as in Lemma~\ref{padicle}, we get
	\[
	|x_{k+1}^{\astmult}y_{k}^{\astmult} - x_{k}^{\astmult}y_{k+1}^{\astmult}|_{p} 
	= |y_{k+1}^{\astmult}(y_{k}^{\astmult} \xi - x_{k}^{\astmult}) 
	- y_{k}^{\astmult}(y_{k+1}^{\astmult} \xi - x_{k+1}^{\astmult})|_{p}
	\leq (Q_{k}^{\astmult})^{-\mu_k}.
	\]
	Thus,
	\[
	|x_{k+1}^{\astmult}y_{k}^{\astmult} - x_{k}^{\astmult}y_{k+1}^{\astmult}| \geq
	(Q_{k}^{\astmult})^{ \mu_k}, 
	\]
	which implies that 
	\[
	\max\{ |x_{k+1}^{\astmult}y_{k}^{\astmult}| , |x_{k}^{\astmult}y_{k+1}^{\astmult}| \}  \gg (Q_{k}^{\astmult})^{ \mu_k}.
	\]
	Set $\delta_{k}= \min\{ \alpha_{k}, \beta_{k} \}$. Since $|x_k^{\astmult}| \gg |y_k^{\astmult}|$, we may 
	assume that $\delta_{k}=  \beta_{k}$ (if necessary, we absorb the numerical constant in $\ll$). 
	We have either
	\[
	|x_{k+1}^{\astmult}| \gg
	\frac{(Q_{k}^{\astmult})^{ \mu_k} }{|y_{k}^{\astmult}|}= 
	(Q_{k}^{\astmult})^{ \mu_k-\beta_{k}}, 
	\]
	which gives
	\[
	Q_{k+1}^{\astmult} = \sqrt{ |x_{k+1}^{\astmult} y_{k+1}^{\astmult} | } \geq
	\sqrt{|x_{k+1}^{\astmult}|} \gg  (Q_{k}^{\astmult})^{ \frac{\mu_k -\delta_{k} }{2}  }, 
	\]
	or
	\[
	|x_{k+1}^{\astmult}| \gg |y_{k+1}^{\astmult}| \gg 
	\frac{(Q_{k}^{\astmult})^{ \mu_k} }{|x_{k}^{\astmult}|} = (Q_{k}^{\astmult})^{ \mu_k -\alpha_{k} } \gg 
	 (Q_{k}^{\astmult})^{ \mu_k- 2+\delta_{k}}, 
	\]
	which gives
	\[
	Q_{k+1}^{\astmult} = \sqrt{ |x_{k+1}^{\astmult} y_{k+1}^{\astmult} | } \geq
	 (Q_{k}^{\astmult})^{ \mu_k-2+\delta_{k}}.
	\]
	To sum up, we have proved that 
	\[
	\frac{ \log Q_{k+1}^{\astmult} }{ \log Q_{k}^{\astmult} } \geq 
	\min \Bigl\{  \frac{\mu_k-\delta_{k} }{2} , \mu_k - 2 +\delta_{k} \Bigr\}.
	\]
	Let $\eps > 0$ be a given real number. 
	For $k$ large enough, it then follows from the proof of Lemma~\ref{hilfl} that 
	\begin{equation}  \label{eq:mutmu}
	\frac{ \log Q_{k+1}^{\astmult} }{ \log Q_{k}^{\astmult} } \leq 
	\frac{\mu_k - 1} { \hmu^{\astmult}(\xi) - 1 } + \eps.
	\end{equation}
	We deduce that 
	\begin{align*}
	\hmu^{\astmult}(\xi) & \leq 1 + \max  \Bigl\{  \frac{2\mu_k - 2}{\mu_k -\delta_{k}}, 
	\frac{\mu_k - 1}{\mu_k-2+\delta_{k}} \Bigr\} + \eps \\
	& \leq 1 + \max  \Bigl\{  \frac{2\mu_k - 2}{\mu_k -1}, 
	\frac{\mu_k - 1}{\mu_k-2} \Bigr\} + \eps
	= 1 + \max  \Bigl\{  2, 
	\frac{\mu_k - 1}{\mu_k-2} \Bigr\} + \eps.
	\end{align*}
	If $\mu_k \le 3$ for arbitrarily large $k$ as above, then the upper bound
	$\hmu^{\astmult}(\xi) \leq  3$ follows from \eqref{eq:mutmu}. 
	Otherwise, we have $\mu_k > 3$ for every sufficiently large $k$ and, 
	since $\eps$ can be taken arbitrarily small, we conclude that 
	$\hmu^{\astmult}(\xi) \leq  3$, as asserted. 
        \end{proof}

\begin{proof}[Proof of the second assertion \eqref{eq:jnik} of Theorem~\ref{endlich}]
First, note that \eqref{eq:jnik} clearly holds  when $\hmu^{\astmult}(\xi)  \leq 3$, since we then have 
$$
\hmu^{\astmult}(\xi)^2 - 3 \hmu^{\astmult}(\xi)  + 3 \leq \hmu^{\astmult}(\xi)  \leq \mu^{\astmult}(\xi).
$$
Consequently, we assume throughout this proof 
that $\hmu^{\astmult}(\xi) > 3$. By \eqref{eq:andere}, we then have 
$$
3 < \mu^{\astmult}(\xi) < + \infty.
$$
Observe also that \eqref{eq:jnik} can be rewritten as
\begin{equation}  \label{eq:bb}
	\htmu(\xi)\leq \frac{3+\sqrt{4\mu^{\astmult}(\xi)-3}}{2}.
	\end{equation}

         Define $\mu_k$ by
	\[
	|x_{k}^{\astmult}\xi-y_{k}^{\astmult}|_{p} = 
	\Bigl( \sqrt{|x_{k}^{\astmult}y_{k}^{\astmult}| } \Bigr)^{-\mu_k}=
	(Q_{k}^{\astmult})^{-\mu_k}, 
	\]
	 and $\alpha_{k}, \beta_{k}, \delta_{k}$ by
	\[
	(Q_{k}^{\astmult})^{ \alpha_{k}}= |x_{k}^{\astmult}|,\quad (Q_{k}^{\astmult})^{ \beta_{k}}= |y_{k}^{\astmult}|, \quad \delta_{k}=\min\{ \alpha_{k},\beta_{k} \}.
	\]
	We can assume that    
	$$ 
	\mu^{\astmult}(\xi)= \limsup_{\ell \to \infty} \, \mu_{2 \ell},
	$$ 
	and, in view of Theorem~\ref{neu},  that for all large even integers $k$ we have 
	$$ 
	|x_{k}^{\astmult}| > |y_{k}^{\astmult}|, \quad |x_{k+1}^{\astmult}| < |y_{k+1}^{\astmult}|. 
	$$ 
    Then $\delta_k = \beta_k$. Below, $k$ denotes a sufficiently large even integer.
	
	Let $\eps > 0$ be a given real number. 	Proceeding as in the preceding proof, 
	but with the pairs $(x_{k}^{\astmult}, y_{k}^{\astmult})$
	and $(x_{k+2}^{\astmult}, y_{k+2}^{\astmult})$ which 
	satisfy the inequalities $|x_{k+2}^{\astmult} | > |y_{k+2}^{\astmult}|$ and 
	$|x_{k}^{\astmult}| > |y_{k}^{\astmult}|$, we get
	\[
	\frac{ \log Q_{k+2}^{\astmult} }{ \log Q_{k}^{\astmult} } \geq 
	\min\left\{  \frac{\mu_k -\delta_{k} }{2} ,\; \mu_k-2+\delta_{k}  \right\} + \eps,
	\]
	for $k$ large. 
	On the other hand, from the proof of Lemma~\ref{hilfl} we get 
	\[
	\frac{ \log Q_{k+2}^{\astmult} }{ \log Q_{k}^{\astmult} } = 
	\frac{ \log Q_{k+2}^{\astmult} }{ \log Q_{k+1}^{\astmult} } \cdot \frac{ \log Q_{k+1}^{\astmult} }{ \log Q_{k}^{\astmult} } \leq
	\left(\frac{\mu_k - 1}{ \hmu^{\astmult}(\xi) - 1}\right)^{2}+ \eps. 
	\]
	The combination of the latter inequalities gives
	\[
	\hmu^{\astmult}(\xi)  \leq 1 +  
	\sqrt{ \max\Bigl\{ \frac{2(\mu_k-1)^{2} }{\mu_k -\delta_{k} }, 
		\frac{ (\mu_k-1)^{2} }{\mu_k-2+\delta_{k}   }  \Bigr\} } + \tilde{\eps},
	\]
	where $\tilde{\eps}$ tends to $0$ with $\eps$. 
	By considering a sequence $(k_j)_{j \ge 1}$ of even integers such that
	$$
	\mu^{\astmult}(\xi)= \limsup_{j \to \infty} \, \mu_{k_j}
	$$ 
	and putting
	$$
	\overline{\delta}= \limsup_{j \to \infty} \delta_{k_j},
	$$
	we deduce that 
	$$
	\hmu^{\astmult}(\xi)  \leq 1 +  
	  \max\biggl\{ \sqrt{ \frac{2}{\mu^{\astmult}(\xi) - \overline{\delta} } } (\mu^{\astmult}(\xi)-1), 
		\frac{ \mu^{\astmult}(\xi) -1}{\sqrt{\mu^{\astmult}(\xi) -2+\overline{\delta} }}   \biggr\}. 
	$$
	Observe that in the maximum the right hand term is larger than 
	the left hand term if and only if $\mu^{\astmult}(\xi) \le 4 - 3 \overline{\delta}$. 
	Consequently, if $\mu^{\astmult}(\xi) \le 4 - 3 \overline{\delta} \leq 4$, then we get 
	\[
	\hmu^{\astmult}(\xi) \leq  1 + \frac{\mu^{\astmult}(\xi) - 1}{ \sqrt{\mu^{\astmult}(\xi)-2}}.
	\]
         Taking into account that $\mu^{\astmult}(\xi) > 3$, a rapid 
         calculation shows that this inequality implies \eqref{eq:bb}, as wanted.

	So we may assume $\mu^{\astmult}(\xi) > 4 - 3 \overline{\delta}$ and thus
	$$
	\hmu^{\astmult}(\xi)  \leq 1 + \sqrt{ \frac{2}{\mu^{\astmult}(\xi) - \overline{\delta} } } (\mu^{\astmult}(\xi)-1). 
	$$
	Observe that Corollary~\ref{KOR} applied with $A =2-\overline{\delta}$ gives
	$$
	\hmu^{\astmult}(\xi) \leq \max \left\{ \frac{3\mu^{\astmult}(\xi) - 2 - A}{\mu^{\astmult}(\xi)-A}, \;  
\frac{2 \mu^{\astmult}(\xi)}{\mu^{\astmult}(\xi) -A}  \right\},  
	$$
	where the maximum 
	is given by the left hand term if and only if we have $A\leq \mu^{\astmult}(\xi)-2$.
	We distinguish two cases. 
\\

		Case 1: Assume that $A\leq \mu^{\astmult}(\xi)-2$, that is, 
		$\mu^{\astmult}(\xi)\geq 4-\overline{\delta}$.
Then
\begin{equation}  \label{eq:J}
		 \htmu(\xi) \leq \min\left\{ \frac{3\mu^{\astmult}(\xi) -4+ \overline{\delta} }{ \mu^{\astmult}(\xi) -2+\overline{\delta} } \;,\;  \sqrt{\frac{2}{\mu^{\astmult}(\xi)- \overline{\delta} }} (\mu^{\astmult}(\xi)-1)+1 \right\}
		 \end{equation}
		 holds. In view of \eqref{eq:bb}, we can assume that 
	 \[
	\frac{3\mu^{\astmult}(\xi) -4 +   \overline{\delta} }{\mu^{\astmult}(\xi) - 2 +   \overline{\delta} } 
	>  \frac{3+\sqrt{4\mu^{\astmult}(\xi)-3}}{2},
	\]
	that is,
	\[
	 \overline{\delta}  
	  < \frac{3 \mu^{\astmult}(\xi) - 2 - (\mu^{\astmult}(\xi) - 2)\sqrt{4\mu^{\astmult}(\xi) - 3}}{1 + \sqrt{4\mu^{\astmult}(\xi) - 3}}.
	 \]
	 Using this bound for $\overline{\delta}$, we derive from \eqref{eq:J} that 
	 \[
	 \htmu(\xi) < 1 + \sqrt{\frac{(\mu^{\astmult}(\xi) - 1) 
	 ( 1 + \sqrt{4\mu^{\astmult}(\xi) - 3})}{\sqrt{4\mu^{\astmult}(\xi) - 3} - 1} }
	 = 1 +  \frac{1+\sqrt{4\mu^{\astmult}(\xi)-3}}{2},
	 \]
	 which gives the bound \eqref{eq:bb}. 
	 \\
	 
	Case 2:    We assume that $A>\mu^{\astmult}(\xi)-2$, that is, 
	$\mu^{\astmult}(\xi)< 4-\overline{\delta}\leq 4$. 
 Then \eqref{eq:hatbis} gives
	\[
	\htmu(\xi) \leq 2+\frac{2A }{\mu^{\astmult}(\xi) -A } = \frac{2\mu^{\astmult}(\xi) }{ \mu^{\astmult}(\xi) -2+\overline{\delta} }
	\]
	and we get
	 \begin{equation} \label{eq:Jbis}
	\htmu(\xi) \leq \min\left\{ \frac{2\mu^{\astmult}(\xi) }{ \mu^{\astmult}(\xi) -2+\overline{\delta} } \;,\;  \sqrt{\frac{2}{\mu^{\astmult}(\xi)- \overline{\delta} }} (\mu^{\astmult}(\xi)-1)+1 \right\}.
	\end{equation}
        In view of \eqref{eq:bb}, we can assume that
	 \[
	\frac{2\mu^{\astmult}(\xi) }{ \mu^{\astmult}(\xi) -2+\overline{\delta} }
	>  \frac{3+\sqrt{4\mu^{\astmult}(\xi)-3}}{2},
	\]
that is,
	\[
	 \overline{\delta}   
	 <  \frac{\mu^{\astmult}(\xi) + 6 - (\mu^{\astmult}(\xi) - 2) \sqrt{4\mu^{\astmult}(\xi) - 3}}{ 3 + \sqrt{4\mu^{\astmult}(\xi) - 3}}.
	 \]
	 Using this bound for $\overline{\delta}$, we derive from \eqref{eq:Jbis} that 
	 \[
	 \htmu(\xi) < 1 + \sqrt{\frac{ (\mu^{\astmult}(\xi) - 1)^2
	 ( 3 + \sqrt{4\mu^{\astmult}(\xi) - 3})}
	 {(\mu^{\astmult}(\xi) - 3) + (\mu^{\astmult}(\xi) - 1)\sqrt{4\mu^{\astmult}(\xi) - 3} } }.
	\]
	A careful computation  
	shows that, since $\mu^{\astmult}(\xi) \ge 3$, we get \eqref{eq:bb}. 
	\end{proof}

As noticed above Corollary \ref{cor:Liouv}, the upper bound \eqref{eq:bndere} 
follows from \eqref{eq:andere} and \eqref{eq:jnik}. The proof of Theorem~\ref{endlich} is complete.

\end{document}